\documentclass[12pt,reqno]{amsart}
\usepackage{amsmath, amsthm, amssymb, cite, bm, eufrak,enumitem,bbm,color}

\advance \topmargin by -\headheight
\advance \topmargin by -\headsep
     
\evensidemargin \oddsidemargin
\newtheorem{theorem}{Theorem}[section]
\newtheorem{lemma}[theorem]{Lemma}
\newtheorem{proposition}[theorem]{Proposition}
\newtheorem{corollary}[theorem]{Corollary}

\theoremstyle{definition}

\theoremstyle{remark}

\numberwithin{equation}{section}

\newcommand{\abs}[1]{\lvert#1\rvert}
\newcommand{\mmod}[1]{\,\,(\text{mod}\,\,#1)}

\title{Efficient congruencing in ellipsephic sets:\\the general case}
\author{Kirsti D. Biggs}
\address{Mathematical Sciences, University of Gothenburg and Chalmers University of Technology, 412 96 Gothenburg, Sweden}
\email{biggs@chalmers.se}
\subjclass[2010]{{11A63, 11D45, 11L07, 11P55}}
\keywords{\hspace{-0.01in}Hardy--Littlewood\hspace{-0.01in} method, \hspace{-0.005in}efficient\hspace{-0.01in} congruencing,\hspace{-0.005in} missing\hspace{-0.01in} digits}
\thanks{This paper is based on work appearing in the author's PhD thesis \cite{mythesis} at the University of Bristol, UK, and supported by EPSRC Doctoral Training Partnership EP/M507994/1}

\begin{document}

\begin{abstract}
In this paper, we bound the number of solutions to a general Vinogradov system of equations
\begin{equation*}
x_1^j+\dots+x_s^j=y_1^j+\dots+y_s^j,\quad (1\leq j\leq k),
\end{equation*}
as well as other related systems, in which the variables are required to satisfy digital restrictions in a given base. Specifically, our sets of permitted digits have the property that there are few representations of a natural number as sums of elements of the digit set---the set of squares serving as a key example. We obtain better bounds using this additive structure than could be deduced purely from the size of the set of variables. In particular, when the digits are required to be squares, we obtain diagonal behaviour with $2k(k+1)$ variables.
\end{abstract}

\maketitle

\section{Introduction}
We consider, for a fixed integer $k\in\mathbb{N}$, the system of Diophantine equations
\begin{equation}\label{VMVT}
x_1^j+\dots+x_s^j=y_1^j+\dots+y_s^j,\quad (1\leq j\leq k).
\end{equation}
In \cite{ellipsephic2}, the author proved an upper bound, in the case $k=2$, for the number of solutions to (\ref{VMVT}), with $1\leq x_i,y_i\leq X$ for all $i$, where the variables are restricted to subsets of the natural numbers defined by digital restrictions. In this paper, we extend such results to the case of general $k$.

Fix an odd prime $p>k$, and a subset $A\subset\mathbb{N}_0=\mathbb{N}\cup\{0\}$ with the property that 
\begin{equation}\label{keyAP}
\#\{(a_1,\dots,a_t)\in A^t\mid a_1+\dots+a_t=n\}\ll n^{\epsilon}
\end{equation}
for some $t\geq 2$ and for all $\epsilon>0$, 
and let
\begin{equation*}
\mathcal{E}=\mathcal{E}_p^{A}=\{n\in\mathbb{N}\mid \textstyle{n=\sum_i a_ip^i}, a_i\in A\cap[0,p-1]\,\,\mbox{for all }i\}
\end{equation*}
be the set of natural numbers whose expansion in base $p$ includes only digits from $A$. Write $A_p$ for $A\cap[0,p-1]$, and assume that $2\leq \#A_p\leq p-1$. Let $I_{s,k}(X)$ be the number of solutions to the Vinogradov system (\ref{VMVT}) with $x_i,y_i\in\mathcal{E}(X)=\mathcal{E}\cap[1,X]$ for all $i$, and write $Y$ for $\#\mathcal{E}(X)$. 
\begin{theorem}\label{basicthm}
For all $\epsilon>0$ we have
\begin{equation*}
I_{s,k}(X)\ll X^{\epsilon}(Y^s+Y^{2s-tk(k+1)/2}).
\end{equation*}
\end{theorem}
The implicit constant may depend on the various parameters $s,k,t,p$ and $\epsilon$, but not on $X$. We note that this bound is essentially optimal provided that $Y\gg X^{1/t}$, since one may apply a standard method, discussed later in this section, to see that
\begin{align}\label{lowerbd}
I_{s,k}(X)&\gg Y^{2s}X^{-k(k+1)/2}\gg Y^{2s-tk(k+1)/2},
\end{align}
by our assumption on the size of $Y$, and the bound $I_{s,k}(X)\gg Y^s$ comes from the diagonal solutions.

For historical reasons, upper bounds for the number of solutions to (\ref{VMVT}) go by the name of Vinogradov's mean value theorem---in \cite{wooleyk3}, Wooley used the efficient congruencing method to prove an optimal upper bound for the number of solutions to this system in the case $k=3$, the first time such a bound had been obtained for any $k>2$. In \cite{BDG}, Bourgain, Demeter and Guth proved the equivalent statement for $k\geq 4$ using the harmonic analytic technique of $l^2$-decoupling, often seen as a real analogue of the $p$-adic efficient congruencing. Subsequently, Wooley developed the nested version of his method and used it to provide an alternative proof of the general case in \cite{NEC}. The similarities between the two methods are analysed further in \cite{ecvl2}.

As discussed in \cite{ellipsephic2}, we call our sets with digital restrictions ellipsephic, after the French term \emph{ellips\'ephique}, coined by Mauduit to refer to integers with missing digits, and used, for example, in \cite{Aloui} and \cite{AlouiMMk}. We let $r=\#A_p$, and note that the restriction that $2\leq r\leq p-1$ stems from the fact that the cases $r=0$ and $A_p=\{0\}$ are trivial, and the case $r=p$ reduces to the classical case, while the case $r=1$ (with $A_p\neq\{0\}$) is sufficiently unusual that we omit it from consideration---the number of such `repdigits' is $O(\log{X})$, which makes them much sparser than the ellipsephic sets we consider here. We observe that 
\begin{align*}
\#\mathcal{E}(X)\ll r^{\log_p{X}+1}=rX^{\log_p{r}},
\end{align*}
and hence that $\mathcal{E}$ is a thin set, in the sense that
\begin{equation*}
\lim_{X\to\infty}\frac{\#\mathcal{E}(X)}{X}=0.
\end{equation*}
The effect of these digital restrictions is to give ellipsephic sets a fractal-like structure similar to those seen in the middle-third Cantor set and generalisations thereof. In \cite{LabaPram}, \L aba and Pramanik study maximal operators corresponding to certain real fractal subsets constructed in a similar manner.

We recall the details of the key additive property that we require of our digit set. For an integer $t\geq 2$, we refer to $A\subset\mathbb{N}_0$ as an $E_t^*$-set if (\ref{keyAP}) holds for all $\epsilon>0$. As mentioned in \cite{ellipsephic2}, we can view such sets as a generalisation of Sidon sets, in which the number of representations of an integer as the sum of a fixed number of elements of our set is bounded by a constant. 

Landau proved in \cite{Landau1912} that the set of squares is an $E_2^*$-set, and Hardy and Littlewood conjectured in \cite[Hypothesis K]{PNVI} that for all $k\geq 2$, the set of $k$th powers should be an $E_k^*$-set. However, in \cite{mahlerHypK} Mahler proved that this conjecture is false for the set of cubes, and it remains open to date for $k\geq 4$. Nevertheless, in \cite{Vurandom}, Vu used a probabilistic argument to demonstrate that for any $k\geq 2$, there exists a subset $S_k$ of the set of $k$th powers and an integer $t_k$ such that $S_k$ is an $E_{t_k}^*$-set, thus proving the existence of infinitely many sets of the form we are interested in. The explicit determination of further sets of this form would be of great interest.

We refer to $\mathcal{E}=\mathcal{E}^A_p$ as a $(p,t)^*$-ellipsephic set if $A$ is an $E_t^*$-set, and introduce some further notation to allow us to state the more general form of our main result. Consider a system of polynomials $\bm{\phi}\in\mathbb{Z}[z]^k$ with non-vanishing Wronskian
\begin{align*}
W(z,\bm{\phi})&=\det\big(\phi_j^{(i)}(z)\big)_{1\leq i,j\leq k},
\end{align*}
where $\phi_j^{(i)}(z)$ is the $i$th derivative of $\phi_j$ with respect to $z$. For a sequence $\bm{\mathfrak{a}}=(\mathfrak{a}_x)_{x\in\mathcal{E}}$ of complex weights, we let
\begin{equation*}
J_{s,k}(X)=J_{s,k}(X;\bm{\mathfrak{a}},\bm{\phi})=\oint \Big|\sum_{x\in\mathcal{E}(X)}\mathfrak{a}_x e\big(\alpha_1 \phi_1(x)+\dots+ \alpha_k \phi_k(x)\big)\Big|^{2s}\,d\bm{\alpha},
\end{equation*}
where we write $e(z)$ for $e^{2\pi iz}$ and $\oint$ for the integral over the $k$-dimensional unit cube $[0,1]^k$. Then $J_{s,k}(X)$ counts the solutions $x_i,y_i\in\mathcal{E}(X)$ to the system
\begin{equation*}
\phi_j(x_1)+\dots+\phi_j(x_s)=\phi_j(y_1)+\dots+\phi_j(y_s),\quad (1\leq j\leq k),
\end{equation*}
with weights $\mathfrak{a}_{\bm{x}}\overline{\mathfrak{a}_{\bm{y}}}=\mathfrak{a}_{x_1}\dots\mathfrak{a}_{x_s}\overline{\mathfrak{a}_{y_1}\dots\mathfrak{a}_{y_s}}$. We adopt the convention that, unless previously fixed, statements involving $\epsilon$ hold for any suitably small choice of $\epsilon>0$, and as such the exact value may change from line to line. Additionally, implicit constants in Vinogradov's $\ll$ notation may always depend on $\epsilon$. The vector notation $\bm{x}\equiv\xi\mmod{q}$ means that $x_i\equiv\xi\mmod{q}$ for all $i$, and $\bm{x}\equiv\bm{y}\mmod{q}$ means that $x_i\equiv y_i\mmod{q}$ for all $i$.

Our main theorem provides the following upper bound for $J_{s,k}(X)$.
\begin{theorem}\label{Thm1.1} 
For $k\in\mathbb{N}$, let $\bm{\phi}\in\mathbb{Z}[z]^k$ be a system of polynomials with $W(z,\bm{\phi})\neq 0$. For $t\in\mathbb{N}$ with $t\geq 2$, and for $p>k$ an odd prime sufficiently large in terms of the coefficients of $\bm{\phi}$, let $\mathcal{E}$ be a $(p,t)^*$-ellipsephic set, and write $Y=\#\mathcal{E}(X)$. Then for $s\geq tk(k+1)/2$, we have
\begin{equation*}
J_{s,k}(X)\ll Y^{s-tk(k+1)/2}X^{\epsilon}\bigg(\sum_{x\in\mathcal{E}(X)}\left|\mathfrak{a}_x \right|^2\bigg)^s.
\end{equation*}
In the case where $\bm{\phi}$ is the Vinogradov system, it suffices to take $p>k$ odd.
\end{theorem}
An application of H\"older's inequality shows that for $s\leq tk(k+1)/2$, we have
\begin{align*}
J_{s,k}(X)&\ll 1\cdot J_{tk(k+1)/2,k}(X)^{2s/tk(k+1)}\\
& \ll X^{\epsilon}\bigg(\sum_{x\in\mathcal{E}(X)}\left|\mathfrak{a}_x \right|^2\bigg)^s.
\end{align*}
On the other hand, if we take $\mathfrak{a}_x=0$ for $x\notin\mathcal{E}$ in the classical version of Vinogradov's mean value theorem, for $s=tk(k+1)/2$ we obtain
\begin{equation*}
J_{s,k}(X)\ll Y^{(t-1)k(k+1)/2}X^{\epsilon}\bigg(\sum_{x\in\mathcal{E}(X)}\left|\mathfrak{a}_x \right|^2\bigg)^s,
\end{equation*}
so we see that, as in the quadratic case, we have achieved a power saving in $Y$ by utilising the specific additive structure of our ellipsephic sets, rather than simply their density.

\begin{corollary}\label{Cor1.2}
Theorem \ref{basicthm} is true.
\end{corollary}
\begin{proof}
This is the case of Theorem \ref{Thm1.1} where $\phi_j(z)=z^j$ for $1\leq j\leq k$, and $\mathfrak{a}_x=1$ for all $x\in\mathcal{E}$.
\end{proof}
The lower bound (\ref{lowerbd}) follows by integrating only over the portion of the unit cube for which we have $\alpha_j\ll X^{-j}$ for $1\leq j\leq k$, as in the classical case of Vinogradov's mean value theorem, and using our additional assumption on the size of $Y$.

An important area for future consideration is the application of the results of this paper to Waring's problem, in which we seek to find $s=s(k)$ such that any $n\in\mathbb{N}$ may be written in the form
\begin{align}\label{WPrep}
n=x_1^k+\dots+x_s^k,
\end{align}
with $x_1,\dots, x_s\in\mathcal{E}$. As in \cite{ellipsephic2}, we are able to prove a lower bound for $N_{s,k}(X)=N_{s,k}^{\mathcal{E}}(X)$, defined as the number of positive integers up to $X$ which have a representation in the form (\ref{WPrep}). We require the same condition on $Y$ as in the lower bound discussed above, without which we would not expect to represent a significant proportion of the integers up to $X$.
\begin{corollary}
For natural numbers $k$ and $t$ with $t\geq 2$, and for $p>k$ an odd prime, let $\mathcal{E}$ be a $(p,t)^*$-ellipsephic set. Assume that $Y=\#\mathcal{E}(X)\gg X^{1/t}$. Then for $s\geq tk(k+1)/2$ we have
\begin{equation*}
N_{s,k}(X)\gg X^{1-\epsilon}.
\end{equation*}
\end{corollary}
\begin{proof}
As in \cite[Corollary 1.5]{ellipsephic2}, we write $R(n)=R_{s,k}^{\mathcal{E}}(n)$ for the number of representations of an integer $n$ as a sum of $s$ $k$th powers of integers from $\mathcal{E}$ and apply Cauchy's inequality to see that
\begin{align*}
\bigg(\sum_{1\leq n\leq X}R(n)\bigg)^2 &\leq N_{s,k}(X)\bigg(\sum_{1\leq n\leq X}R(n)^2\bigg).
\end{align*}
Via Theorem \ref{basicthm}, we obtain the bound
\begin{equation*}
N_{s,k}(X)\gg Y^{t(k+1)/2}X^{(1-k)/2-\epsilon},
\end{equation*}
and then use our assumption on the size of $Y$to deduce that
\begin{equation*}
N_{s,k}(X)\gg X^{1-\epsilon},
\end{equation*}
as required.
\end{proof}

The proof of Theorem \ref{Thm1.1} uses Wooley's nested efficient congruencing method and closely follows the argument of \cite{NEC}, with suitable adjustments for our ellipsephic situation. In particular, we require the prime underlying our congruencing argument to be the same prime used to define our ellipsephic set, and we make use of the straightforward fact that whenever $x,y\in\mathcal{E}$ with $p^{a-1}\leq y< p^a$, we have $p^ax+y\in\mathcal{E}$. In Section \ref{genPrelim} of this paper, we provide preliminary notation and formulate an alternative theorem (Theorem \ref{Thm3.1}), which we prove by induction in the next four sections. Specifically, in Section \ref{basecase}, which is the main point of divergence from the work of Wooley, we use the additive properties of our $(p,t)^*$-ellipsephic sets to prove the base case $k=1$ of Theorem \ref{Thm3.1}, using a ``lifting'' argument similar to that in our previous paper \cite{ellipsephic2}. In Section \ref{hierarchy} we introduce a ``hierarchy'' of small constants to support the rest of the paper, and prove some basic results, and in Section \ref{iteration} we use the inductive hypothesis to prove a series of lemmata which form the backbone of our iteration. In Section \ref{PfThm3.1} we complete the proof of Theorem \ref{Thm3.1}, hypothesising that a certain quantity is strictly greater than zero and deriving a contradiction. Finally, in Section \ref{PfMainThm} we use Theorem \ref{Thm3.1} to deduce Theorem \ref{Thm1.1}.

The author would like to thank Trevor Wooley for suggesting this problem and for his invaluable supervision and encouragement, as well as the anonymous referee for helpful comments.

\section{Preliminaries}\label{genPrelim}
We fix natural numbers $k$ and $t$ with $t\geq 2$, an odd prime $p>k$, and a $(p,t)^*$-ellipsephic set $\mathcal{E}$. The following condition on our system of polynomials serves as a proxy for the non-vanishing of the Wronskian; we deduce the general case in Section \ref{PfMainThm}.

Suppose that the system $\bm{\phi}\in\mathbb{Z}[z]^k$ resembles the Vinogradov system in the sense that, for some suitably large $c\in\mathbb{N}$, we have $\phi_j(z)\equiv z^j\mmod{p^c}$ for $1\leq j\leq k$. We call such a system $p^c$-spaced. Note that it is crucial to our argument that the prime $p$ featured here is the same one used to define our digital restrictions.

For a sequence $\bm{\mathfrak{a}}=(\mathfrak{a}_x)_{x\in\mathcal{E}}$ of complex weights with $\sum_{x\in\mathcal{E}}\abs{\mathfrak{a}_x}<\infty$, we let
\begin{equation*}
\rho_0=\rho_0(\bm{\mathfrak{a}})=\bigg(\sum_{x\in\mathcal{E}}\left|\mathfrak{a}_x \right|^2\bigg)^{1/2},
\end{equation*}
and for $\bm{\alpha}\in [0,1]^k$, we let
\begin{equation*}
f(\bm{\alpha})=f(\bm{\alpha};\bm{\mathfrak{a}})=\rho_0^{-1}\sum_{x\in\mathcal{E}}\mathfrak{a}_x e\big(\psi(x;\bm{\alpha})\big),
\end{equation*}
where $\psi(x;\bm{\alpha})=\alpha_1 \phi_1(x)+\dots+ \alpha_k \phi_k(x)$. Consequently, a bound of the form $$J_{s,k}(X)\ll X^{\Delta}\bigg(\sum_{x\in\mathcal{E}(X)}\left|\mathfrak{a}_x \right|^2\bigg)^s,$$ for some $\Delta>0$, follows directly from one of the form
\begin{equation*}
\oint\left|f(\bm{\alpha})\right|^{2s}\,d\bm{\alpha}\ll X^{\Delta}.
\end{equation*}
As in \cite{ellipsephic2}, this normalisation allows us to assume that every $\mathfrak{a}_x$ is real, non-negative and at most one. We let
\begin{equation*}
\mathbb{D}=\bigg\{\bm{\mathfrak{a}}\biggm| 0\leq{\mathfrak{a}_x}\leq 1\mbox{ for all }x\in\mathcal{E}\mbox{ and }0<\sum_{x\in\mathcal{E}}{\mathfrak{a}_x}<\infty\bigg\},
\end{equation*}
and from now on we work with $\bm{\mathfrak{a}}\in\mathbb{D}$.

We also wish to define the restriction of $f(\bm{\alpha})$ to congruence classes modulo various powers of our chosen prime $p$. For $a\in\mathbb{N}$ and $\xi\in\mathcal{E}(p^a)$, let
\begin{equation*}
\rho_a(\xi)=\bigg(\sum_{\substack{x\in\mathcal{E}\\ x\equiv\xi\mmod{p^a}}}\left|\mathfrak{a}_x \right|^2\bigg)^{1/2}
\end{equation*}
and
\begin{equation}\label{fadefn}
f_a(\bm{\alpha},\xi)=\rho_a(\xi)^{-1}\sum_{\substack{x\in\mathcal{E}\\ x\equiv\xi\mmod{p^a}}}\mathfrak{a}_x e\big(\psi(x;\bm{\alpha})\big).
\end{equation}
For later convenience, for any $\xi$ we interpret $\rho_0(\xi)$ to be $\rho_0$ and $f_0(\bm{\alpha},\xi)$ to be $f(\bm{\alpha})$, and we observe that for $a\in\mathbb{N}$, we have
\begin{equation*}
\sum_{\xi\in\mathcal{E}(p^a)}\rho_a(\xi)^2=\rho_0^2,
\end{equation*}
and for $a,b\in\mathbb{N}$ with $a\leq b$,
\begin{equation*}
\sum_{\substack{\xi'\in\mathcal{E}(p^b)\\\xi'\equiv\xi\mmod{p^a}}}\rho_b(\xi')^2=\rho_a(\xi)^2.
\end{equation*}

Our strategy for counting solutions to the system of equations we are interested in involves studying congruences modulo suitably large powers of $p$, and as such we make use of Wooley's notation
\begin{equation}\label{thisorthog}
\oint_{p^B}F(\bm{\alpha})\,d\bm{\alpha} = p^{-kB}\sum_{1\leq u_1\leq p^B}\dots\sum_{1\leq u_k\leq p^B} F(\bm{u}/p^B),
\end{equation}
and define
\begin{equation}\label{Udefn}
U_{s,k}^B(\bm{\mathfrak{a}})=U_{s,k}^{B,\bm{\phi}}(\bm{\mathfrak{a}})=\oint_{p^B}\left|f(\bm{\alpha})\right|^{2s}\,d\bm{\alpha},
\end{equation}
which counts solutions to the system of congruences
\begin{equation}\label{congruences}
\sum_{i=1}^s\big(\phi_j(x_i)-\phi_j(y_i)\big)\equiv 0\mmod{p^B},\quad(1\leq j\leq k)
\end{equation}
with $\bm{x},\bm{y}\in\mathcal{E}^s$, where each solution is counted with weight $\rho_0^{-2s}\mathfrak{a}_{\bm{x}}{\mathfrak{a}_{\bm{y}}}$. 
We also wish to count solutions to (\ref{congruences}) with further congruence restrictions on our variables, so for $H\in\mathbb{N}$, we let
\begin{equation*}
U_{s,k}^{B,H}(\bm{\mathfrak{a}})=\rho_0^{-2}\sum_{\xi\in\mathcal{E}(p^H)}\rho_H(\xi)^2\oint_{p^B}\left|f_H(\bm{\alpha},\xi)\right|^{2s}\,d\bm{\alpha}.
\end{equation*}
The integral on the right-hand side imposes the additional condition that $\bm{x}\equiv\bm{y}\equiv\xi\mmod{p^H}$, and the solutions are now counted with weight $\rho_H(\xi)^{-2s}\mathfrak{a}_{\bm{x}}{\mathfrak{a}_{\bm{y}}}$.

We observe that, for $H\in\mathbb{N}$, we have
\begin{equation*}
f(\bm{\alpha})=\rho_0^{-1}\sum_{\xi\in\mathcal{E}(p^H)}\rho_H(\xi) f_H(\bm{\alpha},\xi),
\end{equation*}
so, by H\"older's inequality,
\begin{align*}
\left|f(\bm{\alpha})\right|^{2s}
&\leq \rho_0^{-2s}\bigg(\sum_{\xi\in\mathcal{E}(p^H)} 1\bigg)^{s}\bigg(\sum_{\xi\in\mathcal{E}(p^H)} \rho_H(\xi)^{2}\bigg)^{s-1}\sum_{\xi\in\mathcal{E}(p^H)}\rho_H(\xi)^2 \left|f_H(\bm{\alpha},\xi)\right|^{2s}\\
&\ll \rho_0^{-2} q^{sH} \sum_{\xi\in\mathcal{E}(p^H)}\rho_H(\xi)^2 \left|f_H(\bm{\alpha},\xi)\right|^{2s},
\end{align*}
where we have written $q=\#\mathcal{E}(p)$. 
Consequently, we have
\begin{equation}\label{Ubd}
U_{s,k}^{B}(\bm{\mathfrak{a}})\ll q^{sH} U_{s,k}^{B,H}(\bm{\mathfrak{a}}).
\end{equation}
We may now ask for the minimal value of $\lambda$ such that
\begin{equation*}
U_{s,k}^{B}(\bm{\mathfrak{a}})\ll (q^H)^{\lambda+\epsilon} U_{s,k}^{B,H}(\bm{\mathfrak{a}})
\end{equation*}
as uniformly as possible in the various parameters, and observe that the bound $\lambda\leq s$ follows from (\ref{Ubd}).

For $\tau>0$, let $\Phi_{\tau}(B)$ denote the set of systems $\bm{\phi}\in\mathbb{Z}[z]^k$ which are $p^c$-spaced for some $c\geq\tau B$. We deduce from (\ref{Ubd}) that for $\bm{\phi}\in\Phi_{\tau}(B)$, we have
\begin{equation*}
\sup_{\bm{\mathfrak{a}}\in\mathbb{D}}\frac{\log{(U_{s,k}^{B}(\bm{\mathfrak{a}})/U_{s,k}^{B,H}(\bm{\mathfrak{a}}))}}{\log{q^H}}\leq s
\end{equation*}
for all $H\in\mathbb{N}$.

Now consider the particular choice of $\bm{\mathfrak{b}}\in\mathbb{D}$ with $\mathfrak{b}_x=0$ whenever $x\not\equiv 0\mmod{p^H}$. We have $U_{s,k}^{B}(\bm{\mathfrak{b}})=U_{s,k}^{B,H}(\bm{\mathfrak{b}})$, and consequently
\begin{equation*}
\sup_{\bm{\mathfrak{a}}\in\mathbb{D}}\frac{\log{(U_{s,k}^{B}(\bm{\mathfrak{a}})/U_{s,k}^{B,H}(\bm{\mathfrak{a}}))}}{\log{q^H}}\geq 0.
\end{equation*}

Given $s,k\in\mathbb{N}$ and $\tau>0$, we let $H=\lceil B/k\rceil$ and let
\begin{equation*}
\lambda^*(s,k;\tau)=\limsup_{B\to\infty}\sup_{\bm{\phi}\in\Phi_{\tau}(B)}\sup_{\bm{\mathfrak{a}}\in\mathbb{D}}\frac{\log{(U_{s,k}^{B}(\bm{\mathfrak{a}})/U_{s,k}^{B,H}(\bm{\mathfrak{a}}))}}{\log{q^H}},
\end{equation*}
and
\begin{equation}\label{lambdadefn}
\lambda(s,k)=\limsup_{\tau\to 0}\lambda^*(s,k;\tau).
\end{equation}
We then have $0\leq \lambda^*(s,k;\tau)\leq s$ and consequently $0\leq \lambda(s,k)\leq s$.

This leads us to the statement of a key result to be used in the proof of Theorem \ref{Thm1.1}.
\begin{theorem}\label{Thm3.1}
For natural numbers $k$ and $t$ with $t\geq 2$, and for $p>k$ an odd prime, let $\mathcal{E}$ be a $(p,t)^*$-ellipsephic set. Then $\lambda(tk(k+1)/2,k)=0$.
\end{theorem}

As a corollary, we obtain
\begin{corollary}\label{Cor3.2}
For natural numbers $k$ and $t$ with $t\geq 2$, and for $p>k$ an odd prime,  let $\mathcal{E}$ be a $(p,t)^*$-ellipsephic set. Let $\tau>0$ and $\epsilon>0$, and let $B$ be sufficiently large in terms of $k,\tau$ and $\epsilon$. Set $s=tk(k+1)/2$ and $H=\lceil B/k\rceil$. Then for all $\bm{\phi}\in\Phi_{\tau}(B)$ and $\bm{\mathfrak{a}}\in\mathbb{D}$, we have
\begin{equation*}
U_{s,k}^{B}(\bm{\mathfrak{a}})\ll q^{H\epsilon} U_{s,k}^{B,H}(\bm{\mathfrak{a}}).
\end{equation*}
\end{corollary}
\begin{proof}
By the definition of $\lambda^*(s,k;\tau)$, we have, for sufficiently large $B$, the bound
\begin{equation*}
U_{s,k}^{B}(\bm{\mathfrak{a}})\ll (q^{H})^{\lambda^*(s,k;\tau)+\epsilon} U_{s,k}^{B,H}(\bm{\mathfrak{a}}).
\end{equation*}
Allowing $\tau$ to tend to zero and applying Theorem \ref{Thm3.1} gives the result.
\end{proof}

We introduce some final definitions. For $a,b,c,\nu\in\mathbb{N}$, and for $0\leq r\leq k$ and $R=tr(r+1)/2$, we let
\begin{equation*}
K_{a,b,c}^{r,\bm{\phi}}(\bm{\mathfrak{a}};\xi,\eta)=\oint_{p^B}\left|f_a(\bm{\alpha},\xi)^{2R}f_b(\bm{\alpha},\eta)^{2s-2R}\right|\,d\bm{\alpha}
\end{equation*}
and
\begin{equation*}
K_{a,b,c}^{r,\bm{\phi},\nu}(\bm{\mathfrak{a}})=\rho_0^{-4}\sum_{\xi\in\mathcal{E}(p^a)}\sum_{\substack{\eta\in\mathcal{E}(p^b)\\ \xi\not\equiv\eta\mmod{p^{\nu}}}}\rho_a(\xi)^2\rho_b(\eta)^2K_{a,b,c}^{r,\bm{\phi}}(\bm{\mathfrak{a}};\xi,\eta).
\end{equation*}
Note that $K_{a,b,c}^{r,\bm{\phi}}(\bm{\mathfrak{a}};\xi,\eta)$ counts solutions $(\bm{x},\bm{y},\bm{u},\bm{v})\in\mathcal{E}^{2s}$ to the congruences
\begin{equation}\label{Kcong}
\sum_{i=1}^R \big(\phi_j(x_i)-\phi_j(y_i)\big)\equiv\sum_{l=1}^{s-R} \big(\phi_j(u_l)-\phi_j(v_l)\big)\mmod{p^B},\quad (1\leq j\leq k)
\end{equation}
with $\bm{x}\equiv\bm{y}\equiv\xi\mmod{p^a}$ and $\bm{u}\equiv\bm{v}\equiv\eta\mmod{p^b}$, where each solution is counted with weight $\rho_a(\xi)^{-2R}\rho_b(\eta)^{2R-2s}\mathfrak{a}_{\bm{x}}{\mathfrak{a}_{\bm{y}}}{\mathfrak{a}_{\bm{u}}}\mathfrak{a}_{\bm{v}}$.

We are also interested in normalised versions of these mean values, so for $\Delta\geq 0$ we define
\begin{equation}\label{normdK}
\widetilde{K}_{a,b,c}^{r,\bm{\phi},\nu}(\bm{\mathfrak{a}})_{\Delta}=\Bigg(\frac{K_{a,b,c}^{r,\bm{\phi},\nu}(\bm{\mathfrak{a}})}{q^{\Delta H} U_{s,k}^{B,H}(\bm{\mathfrak{a}})}\Bigg)^{\frac{k-1}{r(k-r)}}.
\end{equation}

We now prove some auxiliary results giving bounds on the above-defined mean values. The assumption that $\tau<\delta$ in the following lemma ensures that our system of polynomials sufficiently resembles the Vinogradov system.

\begin{lemma}\label{Lemma4.1}
For $s,k\in\mathbb{N}$ and $p>k$ an odd prime,  let $\mathcal{E}$ be a $(p,t)^*$-ellipsephic set. Let $0<\epsilon<\tau<\delta<1$, and let $B$ be sufficiently large in terms of $s,k$ and $\epsilon$. Set $H=\lceil B/k\rceil$. Then for all $\bm{\phi}\in\Phi_{\tau}(B)$, for all $\bm{\mathfrak{a}}\in\mathbb{D}$, and for all $h\in\mathbb{N}_0$ with $h\leq (1-\delta)H$, we have
\begin{equation*}
U_{s,k}^{B,h}(\bm{\mathfrak{a}}) \ll (q^{H-h})^{\lambda(s,k)+\epsilon}U_{s,k}^{B,H}(\bm{\mathfrak{a}}).
\end{equation*}
\end{lemma}
\begin{proof}
The integral within the definition of $U_{s,k}^{B,h}(\bm{\mathfrak{a}})$ counts solutions to the system of congruences (\ref{congruences}) with $\bm{x},\bm{y}\in\mathcal{E}^s$ and $\bm{x}\equiv\bm{y}\equiv\xi\mmod{p^h}$, with weights $\rho_h(\xi)^{-2s}\mathfrak{a}_{\bm{x}}{\mathfrak{a}_{\bm{y}}}$. As in \cite[Lemma 4.1]{NEC}, we make use of some linear algebra to transform this situation into one in which we have a set of $p^{c+h}$-spaced polynomials 
\begin{equation*}
\Phi_j(z)=z^j+p^{c+h}z^{k+1}\Upsilon_j(z),\quad (1\leq j\leq k),
\end{equation*}
for some $\Upsilon_j\in\mathbb{Z}[z]$, satisfying
\begin{equation*}
\sum_{i=1}^s \Phi_j(x_i)\equiv\sum_{i=1}^s \Phi_j(y_i)\mmod{p^{B-kh}},\quad (1\leq j\leq k)
\end{equation*}
whenever $\bm{x},\bm{y}$ forms a solution to the original system of congruences counted by $U_{s,k}^{B,h}(\bm{\mathfrak{a}})$.

The fact that $h\leq (1-\delta)H$ allows us to assume that $B-kh$ is sufficiently large with respect to $s,k$ and $\epsilon$. Since $\tau<\delta$, this remains true when we apply the definition (\ref{lambdadefn}) to obtain
\begin{equation*}
U_{s,k}^{B-kh}(\bm{\mathfrak{c}}) \ll (q^{H-h})^{\lambda(s,k)+\epsilon}U_{s,k}^{B-kh,H-h}(\bm{\mathfrak{c}}),
\end{equation*}
where $\bm{\mathfrak{c}}$ is an auxiliary set of weights defined by $\mathfrak{c}_u=\mathfrak{a}_{p^hu+\xi}\,e\big(\psi(p^hu+\xi;\bm{\alpha})\big)$. Rearranging, and using the orthogonality given by the definition (\ref{thisorthog}), we obtain the conclusion.
\end{proof}

\begin{lemma}\label{Lemma4.2}
For $s,k\in\mathbb{N}$ and $p>k$ an odd prime,  let $\mathcal{E}$ be a $(p,t)^*$-ellipsephic set. Let $0<\epsilon<\tau<\delta<1$, and let $B$ be sufficiently large in terms of $s,k$ and $\epsilon$. Set $H=\lceil B/k\rceil$ and let $\nu\in\mathbb{N}_0$ and $r\in\mathbb{N}$ with $1\leq r\leq k-1$. Suppose that $0<\Lambda\leq\lambda(s,k)$. Then for all $\bm{\phi}\in\Phi_{\tau}(B)$, for all $\bm{\mathfrak{a}}\in\mathbb{D}$, and for all $a,b\in\mathbb{N}_0$ with $\max\{a,b\}\leq (1-\delta)H$, we have
\begin{equation*}
\widetilde{K}_{a,b,c}^{r,\bm{\phi},\nu}(\bm{\mathfrak{a}})_{ \Lambda}\ll(q^{H})^{\lambda(s,k)-\Lambda+\epsilon}.
\end{equation*}
\end{lemma}
\begin{proof}
As in \cite[Lemma 4.2]{NEC}, this follows from H\"older's inequality, Lemma \ref{Lemma4.1} and the various definitions.
\end{proof}

\section{The base case $k=1$}\label{basecase}
In this section, we use the properties of our $(p,t)^*$-ellipsephic sets to prove that Theorem \ref{Thm3.1} holds in the base case $k=1$. The arguments resemble those used in the author's paper \cite{ellipsephic2}, in which we proved that a similar theorem holds when $k=2$. The following proposition takes the place of \cite[Lemma 5.1]{NEC} in the work of Wooley.
\begin{proposition}\label{Lemma5.1}
For $t\geq 2$ an integer, and $p$ an odd prime,  let $\mathcal{E}$ be a $(p,t)^*$-ellipsephic set. Then $\lambda(t,1)=0$.
\end{proposition}
\begin{proof}
Let $0<\tau<1$, and let $B\in\mathbb{N}$ be sufficiently large in terms of $\tau$. Fix any $\bm{\mathfrak{a}}\in\mathbb{D}$ and any $\phi\in\Phi_{\tau}(B)$, so that by definition we have $\phi(z)=z+p^c \psi(z)$ for some $c\geq\tau B$ and some $\psi\in\mathbb{Z}[z]$. Then $U_{t,1}^{B}(\bm{\mathfrak{a}})$ counts solutions to the congruence
\begin{equation}\label{5.1}
\sum_{i=1}^t\big(\phi(x_i)-\phi(y_i)\big)\equiv 0\mmod{p^B}
\end{equation}
with $\bm{x},\bm{y}\in\mathcal{E}^t$, and where each solution is counted with weight $\rho_0^{-2t}\mathfrak{a}_{\bm{x}}{\mathfrak{a}_{\bm{y}}}$. We may rewrite (\ref{5.1}) in the form
\begin{equation}\label{5.2}
\sum_{i=1}^t\big(x_i+p^c\psi(x_i)\big)\equiv \sum_{i=1}^t\big(y_i+p^c\psi(y_i)\big)\mmod{p^B},
\end{equation}
allowing us to deduce that
\begin{equation}\label{c1cong}
\sum_{i=1}^t x_i\equiv \sum_{i=1}^t y_i \mmod{p^{c_1}},
\end{equation}
where we write $c_1=\min\{B,c\}$. This is effectively a ``free'' condition which was already contained in our original congruence (\ref{5.1}).

We now recall a slightly simplified form of a definition which appeared in \cite{ellipsephic2}.
For $d\in\mathbb{N}$, and for weights $\bm{\mathfrak{b}}$ with $\abs{\mathfrak{b}_x}\leq 1$ for all $x\in\mathcal{E}$ and $0<\sum_{x\in\mathcal{E}}\abs{\mathfrak{b}_x}<\infty$, we define
\begin{equation*}
G_{d}(\bm{\mathfrak{b}})=\oint_{p^{d}}\bigg|\sum_{\substack{\bm{x}\in\mathcal{E}^t}}\mathfrak{b}_{\bm{x}}e\big(\beta(x_1+\dots+x_t)\big)\bigg|^2\,d\beta,
\end{equation*}
which counts solutions to the congruence
\begin{equation}\label{dcong}
\sum_{i=1}^t x_i\equiv \sum_{i=1}^t y_i \mmod{p^{d}},
\end{equation}
with weights $\mathfrak{b}_{\bm{x}}\overline{\mathfrak{b}_{\bm{y}}}$. (In the notation of \cite{ellipsephic2}, this is essentially $G_{0,d}(\bm{0},\bm{\mathfrak{b}})$.)
The following lemma is effectively a special case of  \cite[Lemma 2.2]{ellipsephic2}, which provides the key ``lifting'' step of the process, in which we make use of the $E_t^*$ property of our digit set to raise the power of $p$ used in our congruences. We present an outline of the proof here for completeness, and note that further details are available in \cite{ellipsephic2}.

\begin{lemma}\label{diaglemma}
We have
\begin{equation*}
G_{d}(\bm{\mathfrak{b}})\ll (tp)^{d\epsilon}\sum_{\bm{u}\in\mathcal{E}(p^d)^t}\Big|\sum_{\substack{\bm{x}\in\mathcal{E}^t\\ \bm{x}\equiv\bm{u}\mmod{p^d}}}\mathfrak{b}_{\bm{x}}\,\Big|^2.
\end{equation*}
\end{lemma}
\begin{proof}
As in \cite[Lemma 2.2]{ellipsephic2}, we write 
 \begin{equation*}
 x_i=\sum_{r\geq 0} x_i^{(r)}p^r \mbox{ and } y_i=\sum_{r\geq 0} y_i^{(r)}p^r,
 \end{equation*}
with $x_i^{(r)},y_i^{(r)}\in A_p$ for $1\leq i\leq t$, and bound the number of solutions to (\ref{dcong}) by considering each base $p$ digit in turn. For $h\in\mathbb{Z}$, let 
\begin{equation*}
\mathcal{A}_t(h)=\bigg\{\bm{u}\in A_p^{t}\biggm| \sum_{i=1}^t u_i=h\bigg\},
\end{equation*}
and
\begin{equation*}
\widetilde{\mathcal{A}}_t(h)=\bigg\{(\bm{u},\bm{v})\in A_p^{2t}\biggm| \sum_{i=1}^t (u_i-v_i)=h\bigg\}.
\end{equation*}
Summing the digits of our variables from lowest to highest, we see that a solution of (\ref{dcong}) satisfies
\begin{equation*}
(\bm{x}^{(r)},\bm{y}^{(r)})\in\widetilde{\mathcal{A}}_t(\lambda_r p-\lambda_{r-1}),\quad (0\leq r\leq d-1)
\end{equation*}
for some $1-t\leq \lambda_{0},\dots,\lambda_{d-1}\leq t-1$ reflecting the potential carry-over in our addition, and where we have written $\lambda_{-1}=0$ for convenience. For such a tuple  $\bm{\lambda}=(\lambda_0,\dots,\lambda_{d-1})$, we write 
\begin{equation*}
\bm{\lambda'}=(\lambda_0 p-\lambda_{-1},\dots,\lambda_{d-1} p-\lambda_{d-2}).
\end{equation*}
For brevity, we use the notation $\underline{\bm{u}}$ to denote the tuple $(\bm{u}^{(0)},\dots,\bm{u}^{(d-1)})$,
 and we write
\begin{equation*}
\mathcal{A}_t(\bm{h})=\bigg\{\underline{\bm{u}}\in A_p^{td}\biggm| \bm{u}^{(r)}\in{\mathcal{A}}_t(h_r)\mbox{ for }0\leq r\leq d-1\bigg\}
\end{equation*}
and
\begin{equation*}
\widetilde{\mathcal{A}}_t(\bm{h})=\bigg\{(\underline{\bm{u}},\underline{\bm{v}})\in A_p^{2td}\biggm| (\bm{u}^{(r)},\bm{v}^{(r)})\in\widetilde{\mathcal{A}}_t(h_r)\mbox{ for }0\leq r\leq d-1\bigg\}.
\end{equation*}
We observe that these are the sets of all possible variables with given digit sums, and that any solution of (\ref{dcong}) lies in $\widetilde{\mathcal{A}}_t(\bm{\lambda'})$. Using this notation, we may write
\begin{align*}
G_{d}(\bm{\mathfrak{b}})= \sum_{\bm{\lambda}\in\{1-t,\dots,t-1\}^{d}} \sum_{(\underline{\bm{u}},\underline{\bm{v}})\in\widetilde{\mathcal{A}}_t(\bm{\lambda'})}\sum_{\substack{\bm{x},\bm{y}\in\mathcal{E}^t\\ (\bm{x},\bm{y})\equiv(\bm{u},\bm{v})\mmod{p^d}}}\mathfrak{b}_{\bm{x}}\overline{\mathfrak{b}_{\bm{y}}}.
\end{align*}
Rearranging and applying the triangle inequality and Cauchy's inequality, we remove the dependence on $\bm{\lambda}$ to deduce that
\begin{align*}
G_{d}(\bm{\mathfrak{b}})
&\ll \sum_{(\underline{\bm{u}},\underline{\bm{v}})\in\widetilde{\mathcal{A}}_t(\bm{0})}\sum_{\substack{\bm{x},\bm{y}\in\mathcal{E}^t\\ (\bm{x},\bm{y})\equiv(\bm{u},\bm{v})\mmod{p^d}}}\mathfrak{b}_{\bm{x}}\overline{\mathfrak{b}_{\bm{y}}}\\
&\ll \sum_{0\leq\bm{n}\leq t(p-1)}\bigg(\sum_{\underline{\bm{u}}\in\mathcal{A}_t(\bm{n})}\Big|\sum_{\substack{\bm{x}\in\mathcal{E}^t\\ \bm{x}\equiv\bm{u}\mmod{p^d}}}\mathfrak{b}_{\bm{x}}\,\Big|^2\bigg)\bigg(\sum_{\underline{\bm{u}}\in\mathcal{A}_t(\bm{n})}1\bigg).
\end{align*}
From our initial assumption that $\mathcal{E}$ is a $(p,t)^*$-ellipsephic set, we know that for $\bm{n}=(n_0,\dots,n_{d-1})$ with $0\leq\bm{n}\leq t(p-1)$, we have
\begin{equation*}
\#\mathcal{A}_t(\bm{n})=\#\bigg\{\underline{\bm{u}}\in A_p^{td}\biggm| \sum_{i=1}^t u_i^{(r)}=n_r\mbox{ for }0\leq r\leq d-1\bigg\}
\ll\prod_{r=0}^{d-1} n_r^{\epsilon}\ll (tp)^{d\epsilon},
\end{equation*}
and consequently
\begin{align*}
G_{d}(\bm{\mathfrak{b}})&\ll (tp)^{d\epsilon}\sum_{\bm{u}\in\mathcal{E}(p^d)^t}\Big|\sum_{\substack{\bm{x}\in\mathcal{E}^t\\ \bm{x}\equiv\bm{u}\mmod{p^d}}}\mathfrak{b}_{\bm{x}}\,\Big|^2,
\end{align*}
as claimed. \qedhere
\end{proof}

\par\noindent\textit{Proof of Proposition \ref{Lemma5.1} (continued).}
We now fix the weights $\bm{\mathfrak{b}}$ appearing in Lemma \ref{diaglemma} to be
\begin{equation*}
\mathfrak{b}_x=\rho_0^{-1}\mathfrak{a}_x e(\alpha\phi(x)).
\end{equation*}
Then $G_{c_1}(\bm{\mathfrak{b}})$ encodes the number of solutions to (\ref{c1cong}), counted with weights $\rho_0^{-2t}\mathfrak{a}_{\bm{x}}{\mathfrak{a}_{\bm{y}}} e\Big(\alpha\sum_{i=1}^t\big(\phi(x_i)-\phi(y_i)\big)\Big)$, and consequently we may insert the condition (\ref{c1cong}) into our original congruence in the form
\begin{align*}
U_{t,1}^{B}(\bm{\mathfrak{a}})&=\oint_{p^B}\bigg|\rho_0^{-1}\sum_{x\in\mathcal{E}}\mathfrak{a}_x e(\alpha \phi(x))\bigg|^{2t}\,d\alpha\\
&= \oint_{p^B}G_{c_1}(\bm{\mathfrak{b}})\,d\alpha.
\end{align*}
By Lemma \ref{diaglemma}, we have
\begin{align*}
U_{t,1}^{B}(\bm{\mathfrak{a}})&\ll (tp)^{c_1\epsilon}\sum_{\substack{\bm{u}\in\mathcal{E}(p^{c_1})^t}}\oint_{p^B}\Big|\sum_{\substack{\bm{x}\in\mathcal{E}^t\\ \bm{x}\equiv\bm{u}\mmod{p^{c_1}}}}\mathfrak{b}_{\bm{x}}\,\Big|^2\,d\alpha,
\end{align*}
where the integrand on the right-hand side now imposes the condition $\bm{x}\equiv\bm{y}\equiv\bm{u}\mmod{p^{c_1}}$. The fact that $p^{c_1}$ divides $x_i-y_i$ implies that $p^{c_1}$ divides $\psi(x_i)-\psi(y_i)$ for $1\leq i\leq t$, and substituting this into (\ref{5.2}) gives the congruence
\begin{equation*}
\sum_{i=1}^t x_i\equiv \sum_{i=1}^t y_i \mmod{p^{c_2}},
\end{equation*}
where $c_2=\min\{2c,B\}$. Repeating this process, we eventually reach the point at which our congruence holds modulo $p^{c_j}$ with $c_j=\min\{jc,B\}=B$, and since $c\geq\tau B$, this happens after at most $\lceil\tau^{-1}\rceil$ steps. Now
\begin{align*}
U_{t,1}^{B}(\bm{\mathfrak{a}})&\ll (tp)^{B\epsilon}\sum_{\substack{\bm{u}\in\mathcal{E}(p^{B})^t}} \oint_{p^B}\Big|\sum_{\substack{\bm{x}\in\mathcal{E}^t\\ \bm{x}\equiv\bm{u}\mmod{p^{B}}}}\mathfrak{b}_{\bm{x}}\,\Big|^2\,d\alpha,
\end{align*}
so returning to the definition of the weights $\bm{\mathfrak{b}}$, we obtain 
\begin{align*}
U_{t,1}^{B}(\bm{\mathfrak{a}})
&\ll p^{2B\epsilon}\rho_0^{-2t}\sum_{\substack{\bm{u}\in\mathcal{E}(p^{B})^t}}\oint_{p^B}\bigg|\sum_{\substack{\bm{x}\in\mathcal{E}^t\\\bm{x}\equiv\bm{u}\mmod{p^{B}}}}\mathfrak{a}_{\bm{x}} e\Big(\alpha\sum_{i=1}^t\phi(x_i)\Big)\bigg|^{2}\,d\alpha\\
&= p^{2B\epsilon}\rho_0^{-2t}\sum_{\substack{\bm{u}\in\mathcal{E}(p^{B})^t}}\oint_{p^B}\bigg|\prod_{i=1}^t\rho_B(u_i)f_B({\alpha},u_i)\bigg|^{2}\,d\alpha.
\end{align*}
Using H\"older's inequality twice, we see that
\begin{align*}
U_{t,1}^{B}(\bm{\mathfrak{a}})&\ll p^{2B\epsilon}\rho_0^{-2t}\sum_{\substack{\bm{u}\in\mathcal{E}(p^{B})^t}}\prod_{i=1}^t\rho_B(u_i)^2\bigg(\oint_{p^B}\big|f_B({\alpha},u_i)\big|^{2t}\,d\alpha\bigg)^{1/t}\\
&= p^{2B\epsilon}\rho_0^{-2t}\Bigg(\sum_{\substack{u\in\mathcal{E}(p^{B})}}\rho_B(u)^2\bigg(\oint_{p^B}\big|f_B({\alpha},u)\big|^{2t}\,d\alpha\bigg)^{1/t}\Bigg)^t\\
&\ll p^{2B\epsilon}\rho_0^{-2}\sum_{\substack{u\in\mathcal{E}(p^{B})}}\rho_B(u)^2\oint_{p^B}\big|f_B({\alpha},u)\big|^{2t}\,d\alpha=p^{2B\epsilon} U_{t,1}^{B,B}(\bm{\mathfrak{a}}).
\end{align*}
By adjusting the value of $\epsilon$ as necessary, we may replace $p^{2B\epsilon}$ by $q^{B\epsilon}$, and consequently we deduce that
\begin{align*}
\frac{\log({U_{t,1}^{B}(\bm{\mathfrak{a}})}/ U_{t,1}^{B,B}(\bm{\mathfrak{a}}))}{\log{q^B}}\ll\epsilon
\end{align*}
for any $\epsilon>0$, and hence, using the definition (\ref{lambdadefn}), we find that $\lambda(t,1)=0$ as claimed.
\end{proof}

\section{The hierarchy}\label{hierarchy}
In order to prove Theorem \ref{Thm1.1}, we assume that $\Lambda=\lambda(tk(k+1)/2,k)>0$, and work towards a contradiction. We introduce small positive numbers
\begin{equation}\label{6.2}
0<\epsilon<\tau<\delta<\mu<1,
\end{equation}
which form a hierarchy in the sense that each element is assumed to be small enough in terms of $k,\Lambda$ and the larger parameters in the inequality (\ref{6.2}). We may then choose $B$ large enough, in terms of all of the above, to ensure that, writing $H=\lceil B/k\rceil$, we have
\begin{equation}\label{Ulowerbd}
U_{s,k}^B(\bm{\mathfrak{a}})\geq (q^H)^{\Lambda-\epsilon}U_{s,k}^{B,H}(\bm{\mathfrak{a}}).
\end{equation}
By Lemma \ref{Lemma4.1}, we may assume that for all $h\in\mathbb{N}_0$ with $h\leq (1-\delta)H$, and for all $\bm{\mathfrak{a'}}\in\mathbb{D}$, we have
\begin{equation*}
U_{s,k}^{B,h}(\bm{\mathfrak{a'}})\leq (q^{H-h})^{\Lambda+\epsilon}U_{s,k}^{B,H}(\bm{\mathfrak{a'}}).
\end{equation*}
We fix parameters
\begin{align}\label{nutheta}
\nu=\lceil 4\epsilon H\Lambda^{-1}\rceil \quad\mbox{ and }\quad \theta=\lceil{\mu H}\rceil
\end{align}
for use in the remainder of the paper, and observe that the existence of $\nu$ is dependent on our assumption that $\Lambda>0$. We also suppose that we have chosen $\epsilon$ sufficiently small to ensure that $\nu<\theta$. The following lemmata provide bounds for $U_{s,k}^B(\bm{\mathfrak{a}})$ which allow us to initiate our iterative process in Section \ref{PfThm3.1}.
\begin{lemma}\label{Lemma6.1}
We have $U_{s,k}^B({\bm{\mathfrak{a}}})\ll q^{s\nu} K_{\nu,\nu,c}^{1,\bm{\phi},\nu}(\bm{\mathfrak{a}})$.
\end{lemma}
\begin{proof}
As in \cite[Lemma 6.1]{NEC}, we use the definitions and H\"older's inequality to obtain
\begin{align*}
U_{s,k}^B({\bm{\mathfrak{a}}})\ll U_{s,k}^{B,\nu}({\bm{\mathfrak{a}}})+q^{s\nu} K_{\nu,\nu,c}^{1,\bm{\phi},\nu}(\bm{\mathfrak{a}}).
\end{align*}
By Lemma \ref{Lemma4.1}, and using (\ref{nutheta}), we have
\begin{align*}
U_{s,k}^{B,\nu}({\bm{\mathfrak{a}}})&\ll (q^{H-\nu})^{\Lambda+\epsilon}U_{s,k}^{B,H}({\bm{\mathfrak{a}}})\\
&\ll q^{-2\epsilon H} (q^{H})^{\Lambda-\epsilon}U_{s,k}^{B,H}({\bm{\mathfrak{a}}}),
\end{align*}
and by (\ref{Ulowerbd}), this implies
\begin{align*}
U_{s,k}^{B,\nu}({\bm{\mathfrak{a}}})&\ll q^{-2\epsilon H} U_{s,k}^{B}({\bm{\mathfrak{a}}}),
\end{align*}
so that
\begin{align*}
U_{s,k}^{B}({\bm{\mathfrak{a}}})\ll q^{s\nu} K_{\nu,\nu,c}^{1,\bm{\phi},\nu}(\bm{\mathfrak{a}})
\end{align*}
as claimed.
\end{proof}

\begin{lemma}\label{Lemma6.2}
For $a,b\in\mathbb{N}_0$ with $a\leq b$, and for $w>0$ and $\xi\in\mathcal{E}$, we have
\begin{align*}
\rho_{a}(\xi)^2 \left|f_{a}(\bm{\alpha},\xi)\right|^{2w}\leq q^{w(b-a)}\sum_{\substack{\zeta\in\mathcal{E}(p^{b})\\ \zeta\equiv\xi\mmod{p^a}}}\rho_{b}(\zeta)^2 \left|f_{b}(\bm{\alpha},\zeta)\right|^{2w}.
\end{align*}
\end{lemma}
\begin{proof}
Apply H\"older's inequality exactly as in \cite[Lemma 6.2]{NEC}.
\end{proof}

\begin{lemma}\label{Lemma6.3}
We have $U_{s,k}^B({\bm{\mathfrak{a}}})\ll q^{s\theta} K_{\theta,\theta,c}^{1,\bm{\phi},\nu}(\bm{\mathfrak{a}})$.
\end{lemma}
\begin{proof}
Apply Lemma \ref{Lemma6.2} twice, as in \cite[Lemma 6.3]{NEC}, to obtain
\begin{align*}
 K_{\nu,\nu,c}^{1,\bm{\phi},\nu}(\bm{\mathfrak{a}})\ll q^{s(\theta-\nu)} K_{\theta,\theta,c}^{1,\bm{\phi},\nu}(\bm{\mathfrak{a}}),
\end{align*}
and substitute this into Lemma \ref{Lemma6.1} to see that
\begin{align*}
U_{s,k}^B({\bm{\mathfrak{a}}})\ll q^{s\nu}q^{s(\theta-\nu)} K_{\theta,\theta,c}^{1,\bm{\phi},\nu}(\bm{\mathfrak{a}}) = q^{s\theta} K_{\theta,\theta,c}^{1,\bm{\phi},\nu}(\bm{\mathfrak{a}})
\end{align*}
as required.
\end{proof}

\section{The iterative process}\label{iteration}
Let $k\geq 2$, and suppose that Theorem \ref{Thm3.1} holds for exponents smaller than $k$. In this section, we make use of the inductive hypothesis and provide the key lemmata underlying our iterative process, before completing the proof of the theorem in Section \ref{PfThm3.1}. We begin with a lemma which raises the power of $p$ involved in one of our congruences, at a small cost.
\begin{lemma}\label{Lemma7.1}
Let $a,b,r\in\mathbb{N}$ with $1\leq r\leq k-1$ and $\min\{a,b\}\geq \delta\theta$. Suppose that
\begin{equation*}
ra\leq (k-r+1)b\leq B,
\end{equation*}
and set
\begin{equation*}
b'=\lceil (k-r+1)b/r\rceil.
\end{equation*}
Then $K_{a,b,c}^{r,\bm{\phi},\nu}(\bm{\mathfrak{a}})\ll q^{tk^2\nu}K_{b',b,c}^{r,\bm{\phi},\nu}(\bm{\mathfrak{a}})$.
\end{lemma}
\begin{proof} Following \cite[Lemma 7.1]{NEC}, we focus on $K_{a,b,c}^{r,\bm{\phi}}(\bm{\mathfrak{a}};\xi,\eta)$, in which we may assume that $p^{\gamma}\|(\xi-\eta)$ for some $\gamma<\nu$, and write $\xi-\eta = \omega p^{\gamma}$ with $(\omega,p)=1$. We introduce
\begin{equation*}
B'=(k-r+1)b-ra-(k-r)\gamma,
\end{equation*}
and in the case $B'\leq \nu$, we observe that
\begin{equation*}
(k-r+1)b-ra\leq \nu+(k-r)\gamma\leq k\nu.
\end{equation*}
Consequently, we have
\begin{equation*}
b'-a=\lceil(k-r+1)b/r\rceil-a\leq 1+k\nu/r,
\end{equation*}
and may apply Lemma \ref{Lemma6.2} to obtain 
\begin{align*}
\rho_{a}(\xi)^2 \left|f_{a}(\bm{\alpha},\xi)\right|^{2R}\leq q^{R(b'-a)}\sum_{\substack{\zeta\in\mathcal{E}(p^{b'})\\ \zeta\equiv\xi\mmod{p^a}}}\rho_{b'}(\zeta)^2 \left|f_{b'}(\bm{\alpha},\zeta)\right|^{2R}.
\end{align*}
We now have
\begin{align*}
K&_{a,b,c}^{r,\bm{\phi}}(\bm{\mathfrak{a}};\xi,\eta)=\oint_{p^B}\left|f_a(\bm{\alpha},\xi)^{2R}f_b(\bm{\alpha},\eta)^{2s-2R}\right|\,d\bm{\alpha}\\
&\leq \rho_{a}(\xi)^{-2}q^{R(b'-a)}\sum_{\substack{\zeta\in\mathcal{E}(p^{b'})\\ \zeta\equiv\xi\mmod{p^a}}}\rho_{b'}(\zeta)^2  \oint_{p^B}\left|f_{b'}(\bm{\alpha},\zeta)\right|^{2R}\left|f_b(\bm{\alpha},\eta)^{2s-2R}\right|\,d\bm{\alpha},
\end{align*}
and so 
\begin{align*}
K_{a,b,c}^{r,\bm{\phi},\nu}(\bm{\mathfrak{a}})&=\rho_0^{-4}\sum_{\xi\in\mathcal{E}(p^a)}\sum_{\substack{\eta\in\mathcal{E}(p^b)\\ \xi\not\equiv\eta\mmod{p^{\nu}}}}\rho_a(\xi)^2\rho_b(\eta)^2 K_{a,b,c}^{r,\bm{\phi}}(\bm{\mathfrak{a}};\xi,\eta)\\
&\leq \rho_0^{-4}\sum_{\xi\in\mathcal{E}(p^a)}\sum_{\substack{\eta\in\mathcal{E}(p^b)\\ \xi\not\equiv\eta\mmod{p^{\nu}}}}\rho_b(\eta)^2 q^{R(b'-a)}\sum_{\substack{\zeta\in\mathcal{E}(p^{b'})\\ \zeta\equiv\xi\mmod{p^a}}}\rho_{b'}(\zeta)^2  K_{b',b,c}^{r,\bm{\phi}}(\bm{\mathfrak{a}};\zeta,\eta)\\
&=q^{R(b'-a)} K_{b',b,c}^{r,\bm{\phi},\nu}(\bm{\mathfrak{a}}).
\end{align*}
Finally, since $R=tr(r+1)/2\leq trk/2$, we conclude that $R(b'-a)\leq tk^2\nu$, and so $K_{a,b,c}^{r,\bm{\phi},\nu}(\bm{\mathfrak{a}})\ll q^{tk^2\nu}K_{b',b,c}^{r,\bm{\phi},\nu}(\bm{\mathfrak{a}})$ as required.

When $B'>\nu$, we consider the solutions counted by $K_{a,b,c}^{r,\bm{\phi},\nu}(\bm{\mathfrak{a}};\xi,\eta)$ and observe that we may suppose that the polynomials $\phi_j$ within the definition of $K_{a,b,c}^{r,\bm{\phi},\nu}(\bm{\mathfrak{a}};\xi,\eta)$ have the form
 $\phi_j(z)=z^j+p^c z^{k+1}\psi_j(z)$ for some $\psi_j\in\mathbb{Z}[z]$. We may rewrite (\ref{Kcong}) as
\begin{equation*}
\sum_{i=1}^R \big(\phi_j(x_i)-\phi_j(y_i)\big)\equiv\sum_{l=1}^{s-R} \big(\phi_j(p^bw_l+\eta)-\phi_j(p^bz_l+\eta)\big)\mmod{p^B},\quad (1\leq j\leq k).
\end{equation*}
Taking linear combinations of these congruences, we deduce the existence of integer polynomials $\Psi_j$ for which, writing
\begin{align}\label{7.7}
\Phi_j(z)=z^j+p^cz^{k+1}\Psi_j(z),
\end{align}
we have
\begin{align*}
\sum_{i=1}^R   \big(\Phi_j(x_i-\eta) - \Phi_j(y_i-\eta)  \big) \equiv \sum_{l=1}^{s-R} \big(\Phi_j(p^bw_i) - \Phi_j(p^bz_i)  \big) \mmod{p^B}
\end{align*}
for $1\leq j\leq k$. The system $\bm{\Phi}$ is $p^c$-spaced, and we have
\begin{align}\label{7.8}
\sum_{i=1}^R \Phi_j(x_i-\eta) \equiv \sum_{i=1}^R \Phi_j(y_i-\eta) \mmod{(p^{jb},p^B)}.
\end{align}
Recall that the solutions of interest here, namely those counted by $K_{a,b,c}^{r,\bm{\phi}}(\bm{\mathfrak{a}};\xi,\eta)$, satisfy $\bm{x}\equiv\bm{y}\equiv\xi\mmod{p^a}$, so we may write $x_i=p^au_i+\xi$ and $y_i=p^av_i+\xi$, which leads to the observation that
\begin{align*}
x_i-\eta&=p^au_i+\xi-\eta \\&= p^au_i+\omega p^{\gamma},
\end{align*}
and similarly $y_i-\eta = p^av_i+\omega p^{\gamma}$. Substituting this into (\ref{7.8}), and recalling that we have $B\geq (k-r+1)b$, we deduce that for $k-r+1\leq j\leq k$, we have
\begin{align}\label{7.10}
\sum_{i=1}^R \Phi_j (p^au_i+\omega p^{\gamma})\equiv\sum_{i=1}^R \Phi_j (p^av_i+\omega p^{\gamma}) \mmod{p^{(k-r+1)b}}.
\end{align}
Applying the binomial theorem to (\ref{7.7}), we see that for $1\leq i, l\leq r$, there exist polynomials $\Theta_l(z)\in\mathbb{Z}[z]$, and coefficients
\begin{align*}
\Omega_{il}\equiv\binom{k-r+l}{i}\mmod{p^{c}}
\end{align*}
such that, writing
\begin{align*}
\Upsilon_l(z)=\sum_{i=1}^r\Omega_{il}(\omega p^{\gamma})^{k-r+l-i}z^i+z^{r+1}\Theta_l(z),
\end{align*}
we have
\begin{align*}
\Phi_{k-r+l}(p^ay+\omega p^{\gamma})-\Phi_{k-r+l}(\omega p^{\gamma})=\Upsilon_l(p^ay).
\end{align*}
By our assumption that $p>k$, the coefficient matrix ${\Omega}=(\Omega_{il})_{1\leq i,l\leq r}$ has non-zero determinant modulo $p$, as follows:
\begin{align*}
1!\cdots r! \det{{\Omega}}&\equiv \det\Big(\big(k-j+1)\dots(k-j+1-(i-1)\big)\Big)_{1\leq i,j\leq r} \mmod{p}\\
&=\det\big((k-j+1)^i\big)_{1\leq i,j\leq r}\\
&\equiv \bigg(\prod_{l=1}^r (k-l+1)\bigg) \bigg(\prod_{1\leq i<j\leq r}\big((k-j+1)-(k-i+1)\big)\bigg)\\
&\not\equiv 0\mmod{p}.
\end{align*}
Consequently, ${\Omega}$ has an inverse ${\Omega}^{-1}\mmod{p^{(k-r+1)b}}$ with integer coefficients. We take linear combinations of our congruences once again, in order to replace $\bf{\Upsilon}$ and $\bf{\Theta}$ by ${\Omega}^{-1}\bf{\Upsilon}$ and ${\Omega}^{-1}\bf{\Theta}$ respectively. We may therefore suppose, without loss of generality, that ${\Omega}$ is the $r\times r$ identity matrix. Since $(w,p)=1$, there exists $w^{-1}$ with $w^{-1}w\equiv 1\mmod{p^{(k-r+1)b}}$, and therefore polynomials $\Xi_l\in\mathbb{Z}[z]$ such that whenever (\ref{7.10}) holds, we have
\begin{equation*}
(\omega p^{\gamma})^{k-r}\sum_{i=1}^R(p^a)^l \big(\Psi_l(u_i)-\Psi_l(v_i)\big)\equiv 0\mmod{p^{(k-r+1)b}}\quad (1\leq l\leq r),
\end{equation*}
where $\Psi_l(z)=z^l+p^{a-(k-r)\gamma}\Xi_l(z)$.

Our hierarchy (\ref{6.2}) allows us to ensure that
\begin{equation*}
k\gamma < k\nu\leq \delta a,
\end{equation*}
and therefore we have
\begin{equation*}
a-(k-r)\gamma>(1-\delta)a>\tau B,
\end{equation*}
so the system of polynomials $\bm{\Psi}$ is $p^c$-spaced for some $c>\tau(k-r+1)b$, and satisfies
\begin{equation}\label{Psicong}
\sum_{i=1}^R \Psi_l(u_i)\equiv \sum_{i=1}^R\Psi_l(v_i)\mmod{p^{B'}}\quad (1\leq l\leq r).
\end{equation}
Let $H'=\lceil B'/r\rceil$, and note that
\begin{equation*}
b'-H'\leq a+1+(k-r)\gamma/r.
\end{equation*}
Using the definition of $\nu$ in (\ref{nutheta}), and recalling that we are in the case $B'>\nu$, we see that $B'>4\epsilon H\Lambda^{-1}$, which means that our hierarchy (\ref{6.2}) ensures that $B'$ is sufficiently large in terms of the other parameters. This will allow us to apply the inductive hypothesis, after a few further preparatory steps. Observe that we may write
\begin{align*}
f_a(\bm{\alpha},\xi)=\rho_a(\xi)^{-1}\sum_{\substack{x\in\mathcal{E}\\ x\equiv\xi\mmod{p^a}}}\mathfrak{a}_x e\big(\psi(x;\bm{\alpha})\big)=
\rho_a(\xi)^{-1}\sum_{u\in\mathcal{E}}\mathfrak{c}_u(\bm{\alpha}),
\end{align*}
where $\mathfrak{c}_u(\bm{\alpha})=\mathfrak{a}_{p^au+\xi}\,e\big(\psi(p^au+\xi;\bm{\alpha})\big)$, and that
\begin{align}\label{rhoeq}
\rho_0(\bm{\mathfrak{c} })^2=\sum_{u\in\mathcal{E}}\left|\mathfrak{c}_u(\bm{\alpha})\right|^2
=\sum_{u\in\mathcal{E}}\left|\mathfrak{a}_{p^au+\xi}\right|^2=\rho_a(\xi)^2.
\end{align}
Let
\begin{align*}
g_{\bm{\mathfrak{c}}}(\bm{\alpha},\bm{\beta})= \rho_0(\bm{\mathfrak{c} })^{-1}\sum_{u\in\mathcal{E}}\mathfrak{c}_u(\bm{\alpha})e\big(\beta_1\Psi_1(u)+\dots+\beta_r\Psi_r(u)\big)
\end{align*}
and note that
\begin{align*}
\oint_{p^{B'}} \left|g_{\bm{\mathfrak{c}}}(\bm{\alpha},\bm{\beta})\right|^{2R}\,d\bm{\beta}=U_{R,r}^{B'}(\bm{\mathfrak{c}}).
\end{align*}
Note that $g_{\bm{\mathfrak{c}}}(\bm{\alpha},\bm{0})=f_a(\bm{\alpha},\xi)$, and 
 that $U_{R,r}^{B'}(\bm{\mathfrak{c}})$ counts ellipsephic solutions to the system of congruences (\ref{Psicong}), where each solution is counted with weight 
\begin{align*}
\rho_0&(\bm{\mathfrak{c} })^{-2R}\prod_{i=1}^R \mathfrak{c}_{u_i}(\bm{\alpha}) \overline{\mathfrak{c}_{v_i}(\bm{\alpha})}\\
&= \rho_{a}(\xi)^{-2R}\bigg(\prod_{i=1}^R \mathfrak{a}_{p^au_i+\xi} \overline{\mathfrak{a}_{p^av_i+\xi}}\bigg) e\bigg(\sum_{i=1}^R \big(\psi(p^au_i+\xi;\bm{\alpha})-\psi(p^av_i+\xi;\bm{\alpha})\big)\bigg).
\end{align*}
Consequently,
\begin{align*}
\oint_{p^B} \left|f_a(\bm{\alpha},\xi)\right|^{2R} \left|f_b(\bm{\alpha},\eta)\right|^{2s-2R}\,d\bm{\alpha}
=\oint_{p^B} \oint_{p^{B'}} \left|g_{\mathfrak{c}}(\bm{\alpha},\bm{\beta})\right|^{2R} \left|f_b(\bm{\alpha},\eta)\right|^{2s-2R}\,d\bm{\beta}\,d\bm{\alpha},
\end{align*}
and so
\begin{equation} \label{7.18}
K_{a,b,c}^{r,\bm{\phi}}(\bm{\mathfrak{a}};\xi,\eta)=\oint_{p^B}U_{R,r}^{B'}(\bm{\mathfrak{c}}) \left|f_b({\bm{\alpha},\eta)}\right|^{2s-2R}\,d\bm{\alpha}.
\end{equation}
At this point, we apply the inductive hypothesis, in the form of Corollary \ref{Cor3.2}, to deduce that
\begin{align}\label{7.19}
U_{R,r}^{B'}(\bm{\mathfrak{c}}) &\ll q^{B'\epsilon^2}U_{R,r}^{B',H'}(\bm{\mathfrak{c}})\nonumber\\
&=q^{B'\epsilon^2}(\bm{\mathfrak{a}})\rho_0(\bm{\mathfrak{c}})^{-2}\sum_{\zeta\in\mathcal{E}(p^{H'})}\rho_{H'}(\zeta,\bm{\mathfrak{c}})^2\oint_{p^{B'}}\left|g_{\bm{\mathfrak{c}'}}(\bm{\alpha},\bm{\beta})\right|^{2R}\,d\bm{\beta},
\end{align}
where $\mathfrak{c}'_n(\bm{\alpha})=\mathfrak{c}_n(\bm{\alpha})$ whenever $n\in\mathcal{E}$ is congruent to $\zeta$ modulo $p^{H'}$, and zero otherwise. Letting $\kappa=p^a\zeta+\xi$, we see that
\begin{align*}
\rho_{H'}(\zeta,\bm{\mathfrak{c}})^2=\sum_{y\equiv\zeta\mmod{p^{H'}}}\abs{\mathfrak{a}_{p^ay+\xi}}^2
=\sum_{n\equiv\kappa\mmod{p^{a+H'}}}\abs{\mathfrak{a}_{n}}^2= \rho_{a+H'}(\kappa)^2.
\end{align*}
Using this, and recalling (\ref{rhoeq}), we substitute (\ref{7.19}) into (\ref{7.18}) to obtain
\begin{align}\label{7.20}
&\rho_a(\xi)^2 K_{a,b,c}^{r,\bm{\phi}}(\bm{\mathfrak{a}};\xi,\eta)\nonumber\\
&\ll q^{B'\epsilon^2}\hspace{-0.2in}\sum_{\substack{\kappa\in\mathcal{E}(p^{a+H'})\\\kappa\equiv\xi\mmod{p^a}}}\hspace{-0.1in}\rho_{a+H'}(\kappa)^2 \oint_{p^B}\oint_{p^{B'}}\left|g_{\bm{\mathfrak{c}'}}(\bm{\alpha},\bm{\beta})\right|^{2R} \left|f_b({\bm{\alpha},\eta)}\right|^{2s-2R}\,d\bm{\beta}\,d\bm{\alpha}.
\end{align}
We now observe that 
\begin{equation*}
\oint_{p^{B'}}\left|g_{\bm{\mathfrak{c}'}}(\bm{\alpha},\bm{\beta})\right|^{2R} \,d\bm{\beta}
\end{equation*}
counts ellipsephic solutions to the system (\ref{Psicong}) with the additional restriction that $\bm{u}\equiv\bm{v}\equiv\zeta\mmod{p^{H'}}$, and with each solution counted with weight
\begin{align*}
\rho_0&(\bm{\mathfrak{c}'})^{-2R}\prod_{i=1}^R \mathfrak{c}'_{u_i}(\bm{\alpha}) \overline{\mathfrak{c}'_{v_i}(\bm{\alpha})}\\
&= \rho_{a+H'}(\kappa)^{-2R}\bigg(\prod_{i=1}^R \mathfrak{a}_{p^au_i+\xi} \overline{\mathfrak{a}_{p^av_i+\xi}}\bigg) e\bigg(\sum_{i=1}^R \big(\psi(p^au_i+\xi;\bm{\alpha})-\psi(p^av_i+\xi;\bm{\alpha})\big)\bigg).
\end{align*}
Since $g_{\bm{\mathfrak{c}'}}(\bm{\alpha},\bm{0})=f_{a+H'}(\bm{\alpha},\kappa)$, we may reverse the previous process to obtain
\begin{align}\label{7.21}
\oint_{p^{B}}\oint_{p^{B'}}\left|g_{\bm{\mathfrak{c}'}}(\bm{\alpha},\bm{\beta})\right|^{2R} &\left|f_b({\bm{\alpha},\eta)}\right|^{2s-2R}\,d\bm{\beta}\,d\bm{\alpha}\nonumber\\
&=\oint_{p^{B}}\left|f_{a+H'}(\bm{\alpha},\kappa)\right|^{2R} \left|f_b({\bm{\alpha},\eta)}\right|^{2s-2R}\,d\bm{\alpha}.
\end{align}
Applying Lemma \ref{Lemma6.2}, we see that
\begin{align*}
\rho_{a+H'}(\kappa)^2 \left|f_{a+H'}(\bm{\alpha},\kappa)\right|^{2R}\leq q^{R(b'-a-H')}\sum_{\substack{\xi'\in\mathcal{E}(p^{b'})\\ \xi'\equiv\kappa\mmod{p^{a+H'}}}}\rho_{b'}(\xi')^2 \left|f_{b'}(\bm{\alpha},\xi')\right|^{2R}.
\end{align*}
We substitute this into (\ref{7.21}) and then (\ref{7.20}) to obtain
\begin{align*}
\rho_a(\xi)^2 K_{a,b,c}^{r,\bm{\phi}}(\bm{\mathfrak{a}};\xi,\eta) \ll q^{B'\epsilon^2+R(b'-a-H')}\sum_{\substack{\xi'\in\mathcal{E}(p^{b'})\\ \xi'\equiv\xi\mmod{p^{a}}}}\rho_{b'}(\xi')^2 K_{b',b,c}^{r,\bm{\phi}}(\bm{\mathfrak{a}};\xi',\eta),
\end{align*}
and therefore
\begin{equation*}
K_{a,b,c}^{r,\bm{\phi},\nu}(\bm{\mathfrak{a}})\ll q^{B'\epsilon^2+R(b'-a-H')}K_{b',b,c}^{r,\bm{\phi},\nu}(\bm{\mathfrak{a}}).
\end{equation*}
Finally, we have
\begin{align*}
R(b'-a-H')=tr(r+1)(b'-a-H')/2<tk^2\nu/2,
\end{align*}
and so, using the hierarchy (\ref{6.2}), we see that
\begin{equation*}
K_{a,b,c}^{r,\bm{\phi},\nu}(\bm{\mathfrak{a}})\ll q^{tk^2\nu}K_{b',b,c}^{r,\bm{\phi},\nu}(\bm{\mathfrak{a}}),
\end{equation*}
as required.
\end{proof}

From now on we drop any reference to $\bm{\phi}, \nu$ and $c$ in our notation, since they are assumed to remain fixed. Let $a,b,r\in\mathbb{N}$ satisfy the hypotheses of Lemma \ref{Lemma7.1}, and let $b'=\lceil (k-r+1)b/r\rceil$. We wish to swap the congruences modulo $p^{b'}$ and modulo $p^b$, which will ultimately permit us to iterate our lifting process.
\begin{lemma}\label{Lemma8.1}
For $r\geq 2$, we have
\begin{equation*}
K_{a,b}^r(\bm{\mathfrak{a}})\ll q^{tk^2\nu} K_{b,b'}^{k-r}(\bm{\mathfrak{a}})^{1/(k-r+1)}K_{b',b}^{r-1}(\bm{\mathfrak{a}})^{(k-r)/(k-r+1)}.
\end{equation*}
When $r=1$, we have
\begin{equation*}
K_{a,b}^1(\bm{\mathfrak{a}})\ll q^{tk^2\nu} K_{b,kb}^{k-1}(\bm{\mathfrak{a}})^{1/k}U_{s,k}^{B,b}(\bm{\mathfrak{a}})^{1-1/k}.
\end{equation*}
\end{lemma}
\begin{proof}
As in \cite[Lemma 8.1]{NEC}, we apply H\"older's inequality to obtain
\begin{equation*}
K_{b',b}^r(\bm{\mathfrak{a}})\ll K_{b,b'}^{k-r}(\bm{\mathfrak{a}})^{1/(k-r+1)}K_{b',b}^{r-1}(\bm{\mathfrak{a}})^{(k-r)/(k-r+1)},
\end{equation*}
so when $r\geq 2$ we are done by Lemma \ref{Lemma7.1}. When $r=1$, we observe that $K_{b',b}^0(\bm{\mathfrak{a}})=U_{s,k}^{B,b}(\bm{\mathfrak{a}})$, which gives the claimed result.
\end{proof}

We now bound the normalised version of our mean values, and we write $\widetilde{K}_{a,b}^{r}(\bm{\mathfrak{a}})$ as shorthand for $\widetilde{K}_{a,b}^{r}(\bm{\mathfrak{a}})_{\Lambda}$.

\begin{lemma}\label{Lemma8.2}
For $r\geq 2$, we have
\begin{equation*}
\widetilde{K}_{a,b}^r(\bm{\mathfrak{a}})\ll q^{tk^2\nu} \widetilde{K}_{b,b'}^{k-r}(\bm{\mathfrak{a}})^{1/(k-r+1)}\widetilde{K}_{b',b}^{r-1}(\bm{\mathfrak{a}})^{1-1/r}.
\end{equation*}
When $r=1$, we have
\begin{equation*}
\widetilde{K}_{a,b}^1(\bm{\mathfrak{a}})\ll q^{2tk^2\nu} \widetilde{K}_{b,kb}^{k-1}(\bm{\mathfrak{a}})^{1/k}(q^{-b})^{\Lambda(1-1/k)}.
\end{equation*}
\end{lemma}
\begin{proof}
As in \cite[Lemma 8.2]{NEC}, when $r\geq 2$ we use Lemma \ref{Lemma8.1} and (\ref{normdK}) to conclude that
\begin{equation*}
\widetilde{K}_{a,b}^r(\bm{\mathfrak{a}})\ll (q^{tk^2\nu})^{(k-1)/r(k-r)} \widetilde{K}_{b,b'}^{k-r}(\bm{\mathfrak{a}})^{1/(k-r+1)}\widetilde{K}_{b',b}^{r-1}(\bm{\mathfrak{a}})^{1-1/r},
\end{equation*}
which leads directly to the desired conclusion since $(k-1)/r(k-r)\leq 1$ 
for $1\leq r\leq k-1$. When $r=1$, we have
\begin{equation*}
\widetilde{K}_{a,b}^1(\bm{\mathfrak{a}})\ll q^{tk^2\nu} \widetilde{K}_{b,kb}^{k-1}(\bm{\mathfrak{a}})^{1/k}V^{1-1/k},
\end{equation*}
where
\begin{equation*}
V=\frac{U_{s,k}^{B,b}(\bm{\mathfrak{a}})}{q^{\Lambda H}U_{s,k}^{B,H}(\bm{\mathfrak{a}})}.
\end{equation*}
We also have
\begin{align*}
U_{s,k}^{B,b}({\bm{\mathfrak{a}}})&\ll (q^{H-b})^{\Lambda+\epsilon}U_{s,k}^{B,H}({\bm{\mathfrak{a}}}),
\end{align*}
by Lemma \ref{Lemma4.1}, and consequently
\begin{equation*}
V\ll q^{\epsilon H-b({\Lambda+\epsilon})} \ll q^{s\nu-\Lambda b}.
\end{equation*}
We therefore see that
\begin{align*}
\widetilde{K}_{a,b}^1(\bm{\mathfrak{a}})&\ll q^{tk^2\nu} \widetilde{K}_{b,kb}^{k-1}(\bm{\mathfrak{a}})^{1/k}(q^{s\nu-\Lambda b})^{1-1/k}\\
&\ll q^{2tk^2\nu} \widetilde{K}_{b,kb}^{k-1}(\bm{\mathfrak{a}})^{1/k}(q^{-\Lambda b})^{1-1/k}
\end{align*}
since $s=tk(k+1)/2\leq tk^2$.
\end{proof}

For $1\leq j\leq k-1$, we write $\rho_j=j/(k-j+1)$ and $b_j=\lceil b/\rho_j\rceil$. In the next lemma, we make use of the inductive hypothesis to improve our bound on $\widetilde{K}_{a,b}^r(\bm{\mathfrak{a}})$.
\begin{lemma}\label{Lemma9.1}
Let $1\leq r\leq k-1$, and let $a\geq \delta\theta$ and $b\geq k\delta\theta$ with $ra\leq (k-r+1)b$. Then for $kb\leq B$, we have
\begin{equation*}
\widetilde{K}_{a,b}^r(\bm{\mathfrak{a}})\ll q^{(r+1)tk^2\nu} (q^{-b})^{\Lambda(1-1/k)/r}\prod_{j=1}^r \widetilde{K}_{b,b_j}^{k-j}(\bm{\mathfrak{a}})^{\rho_j/r}.
\end{equation*}
\end{lemma}
\begin{proof}
When $r=1$, this follows immediately from Lemma \ref{Lemma8.2}. For $r\geq 2$, we proceed inductively, as in \cite[Lemma 9.1]{NEC}. Suppose that the conclusion is known for all $r<r_0$ for some $2\leq r_0\leq k-1$. By Lemma \ref{Lemma8.2}, we have
\begin{equation}\label{Kr_0}
\widetilde{K}_{a,b}^{r_0}(\bm{\mathfrak{a}})\ll q^{tk^2\nu} \widetilde{K}_{b,b_0}^{k-r_0}(\bm{\mathfrak{a}})^{1/(k-r_0+1)}\widetilde{K}_{b_0,b}^{r_0-1}(\bm{\mathfrak{a}})^{1-1/r_0},
\end{equation}
with $b_0=b_{r_0}=\lceil (k-r_0+1)b/r_0\rceil\geq 2b/k > \delta\theta$. We also have
\begin{align*}
(r_0-1)b_0 \leq (r_0-1)\big((k-r_0+1)b/r_0+1\big)<(k-r_0+2)b,
\end{align*}
where the second inequality follows from (\ref{nutheta}) and the fact that we may choose $B$ sufficiently large. We therefore use the inductive hypothesis to bound $\widetilde{K}_{b_0,b}^{r_0-1}(\bm{\mathfrak{a}})$, obtaining
\begin{equation*}
\widetilde{K}_{b_0,b}^{r_0-1}(\bm{\mathfrak{a}})\ll q^{r_0tk^2\nu} (q^{-b})^{\Lambda(1-1/k)/(r_0-1)}\prod_{j=1}^{r_0-1} \widetilde{K}_{b,b_j}^{k-j}(\bm{\mathfrak{a}})^{\rho_j/(r_0-1)}.
\end{equation*}
Substituting this into (\ref{Kr_0}), and writing $\rho_0=\rho_{r_0}=r_0/(k-r_0+1)$, we see that
\begin{align*}
\widetilde{K}_{a,b}^{r_0}(\bm{\mathfrak{a}})&\ll q^{tk^2\nu+(r_0-1)tk^2\nu}(q^{-b})^{\Lambda(1-1/k)/r_0} \widetilde{K}_{b,b_0}^{k-r_0}(\bm{\mathfrak{a}})^{\rho_{0}/r_0}\prod_{j=1}^{r_0-1} \widetilde{K}_{b,b_j}^{k-j}(\bm{\mathfrak{a}})^{\rho_j/r_0}\\
&\ll q^{r_0tk^2\nu}(q^{-b})^{\Lambda(1-1/k)/r_0} \prod_{j=1}^{r_0} \widetilde{K}_{b,b_j}^{k-j}(\bm{\mathfrak{a}})^{\rho_j/r_0},
\end{align*}
and so the lemma follows by induction.
\end{proof}

\begin{lemma}\label{Lemma9.2}
Suppose that all of the hypotheses of Lemma \ref{Lemma9.1} hold. Then there exists an integer $r'$ with $1\leq r'\leq r$ such that
\begin{equation*}
\widetilde{K}_{a,b}^r(\bm{\mathfrak{a}})\ll \widetilde{K}_{b,b_{r'}}^{k-r'}(\bm{\mathfrak{a}})^{\rho_{r'}} (q^{-b})^{\Lambda/(2k)}.
\end{equation*}
\end{lemma}
\begin{proof}
As in \cite[Lemma 9.2]{NEC}, we combine the inequality
\begin{equation*}
\left|z_1\dots z_n\right|\leq \left| z_1\right|^n+\dots+\left| z_n\right|^n
\end{equation*}
with Lemma \ref{Lemma9.1} to obtain
\begin{equation*}
\widetilde{K}_{a,b}^r(\bm{\mathfrak{a}})\ll q^{(r+1)tk^2\nu} (q^{-b})^{\Lambda(1-1/k)/r}\sum_{j=1}^r \widetilde{K}_{b,b_j}^{k-j}(\bm{\mathfrak{a}})^{\rho_j}.
\end{equation*}
In particular, for some $1\leq r'\leq r$, we have
\begin{equation*}
\widetilde{K}_{a,b}^r(\bm{\mathfrak{a}})\ll q^{(r+1)tk^2\nu} (q^{-b})^{\Lambda(1-1/k)/r} \widetilde{K}_{b,b_{r'}}^{k-r'}(\bm{\mathfrak{a}})^{\rho_{r'}},
\end{equation*}
so it remains to prove that 
\begin{equation}\label{randomqb}
q^{(r+1)tk^2\nu} (q^{-b})^{\Lambda(1-1/k)/r}\leq (q^{-b})^{\Lambda/(2k)}.
\end{equation}
We have $(1-1/k)/r\geq 1/k$ for $1\leq r\leq k-1$, so
\begin{equation*}
q^{(r+1)tk^2\nu} (q^{-b})^{\Lambda(1-1/k)/r}\leq q^{(r+1)tk^2\nu} (q^{-b})^{\Lambda/k}.
\end{equation*}
By our assumptions on $b$ and $r$, and using (\ref{nutheta}), we see that
\begin{equation*}
b\Lambda/k\geq \delta\theta\Lambda\geq \delta\mu H\Lambda \mbox{ and }2tk^3\nu\geq 2(r+1)tk^2\nu,
\end{equation*}
and by (\ref{6.2}) and (\ref{nutheta}), we may choose our parameters to ensure that
\begin{equation*}
\delta\mu H\Lambda>2tk^3 \lceil 4\epsilon H\Lambda^{-1}\rceil=2tk^3\nu,
\end{equation*}
so
\begin{equation*}
q^{(r+1)tk^2\nu} \leq q^{b\Lambda/(2k)}
\end{equation*}
and (\ref{randomqb}) is proved.
\end{proof}
Finally, we use Lemma \ref{Lemma9.2} to deduce an iterative bound of the necessary shape, which will be used in Section \ref{PfThm3.1} to prove Theorem \ref{Thm3.1}.
\begin{lemma}\label{Lemma9.3}
Let $1\leq r\leq k-1$, and suppose $a\geq \delta\theta$ and $b\geq k^2\delta\theta$ with $ra\leq (k-r+1)b$. Then whenever $k^2b\leq B$, there exist integers $r'$ with $1\leq r'\leq k-1$, as well as $a'\geq \delta\theta$ and $b'\geq k^2\delta\theta$ with $r'a'\leq (k-r'+1)b'$, and there exists a real number $0<\rho\leq (1-1/k)^2$ satisfying
\begin{equation*}
(1+2/k)b\leq b'\leq k^2b, \quad\quad b'=\bigg\lceil\frac{(r'+1)a'}{k-r'} \bigg\rceil, \quad\quad \rho b'\geq b,
\end{equation*}
and such that
\begin{equation*}
\widetilde{K}_{a,b}^r(\bm{\mathfrak{a}})\ll \widetilde{K}_{a',b'}^{r'}(\bm{\mathfrak{a}})^{\rho} (q^{-b})^{\Lambda/(2k)}.
\end{equation*}
\end{lemma}
\begin{proof}
Exactly as in \cite[Lemma 9.3]{NEC}, we apply Lemma \ref{Lemma9.2} twice, and then verify that the conditions hold.
\end{proof}

\section{Proof of Theorem \ref{Thm3.1}}\label{PfThm3.1}
Throughout this section, we consider $k\in\mathbb{N}$ and let $s=tk(k+1)/2$. The case $k=1$ has been handled in Proposition \ref{Lemma5.1}, so we may assume that $k\geq 2$, and that Theorem \ref{Thm3.1} is known for exponents smaller than $k$. If $\lambda(s,k)\leq 0$, we are done, so we assume that $\lambda(s,k)=\Lambda>0$ and work towards a contradiction. As in \cite[Section 10]{NEC}, we use Lemma \ref{Lemma6.3} and our hierarchy (\ref{6.2}) to see that
\begin{equation}\label{10.1}
\widetilde{K}_{\theta,\theta}^1(\bm{\mathfrak{a}})\gg q^{-2s\theta}.
\end{equation}

We now set $N=\lceil 16sk/\Lambda\rceil$, again noting that the existence of $N$ depends on the assumption that $\Lambda>0$, and repeatedly apply Lemma \ref{Lemma9.3} to obtain sequences $(a_n),(b_n),(r_n)$ and $(\rho_n)$ for $0\leq n\leq N$, satisfying
\begin{align*}
1\leq r_n\leq k-1, \quad k^2\delta\theta\leq b_n\leq k^{2n+2}\theta,\quad \delta\theta\leq a_n\leq (k-r_n+1)b_n/r_n,
\end{align*}
and, for $n\geq 1$,
\begin{align*}
0<\rho_n\leq (1-1/k)^2,\quad \rho_n b_n\geq b_{n-1},
\end{align*}
and such that
\begin{equation}\label{10.6}
\widetilde{K}_{\theta,\theta}^1(\bm{\mathfrak{a}})\ll \widetilde{K}_{a_n,b_n}^{r_n}(\bm{\mathfrak{a}})^{\rho_1\dots\rho_n}(q^{-\Lambda/(2k)})^{nb_0},
\end{equation}
where the empty product $\rho_1\dots\rho_n$ for $n=0$ is interpreted as $1$. The initial choice $a_0=b_0=\theta$ and $r_0=\rho_0=1$ therefore trivially satisfies (\ref{10.6}). The existence of such sequences follows by induction, using the same argument as in \cite[Section 10]{NEC}.

Using (\ref{10.6}) in the case $n=N$ in conjunction with (\ref{10.1}), and writing $\rho=\rho_1\dots\rho_N$, gives the bound
\begin{equation}\label{1stpfbd}
q^{-2s\theta}\ll \widetilde{K}_{a_N,b_N}^{r_N}(\bm{\mathfrak{a}})^{\rho}(q^{-\Lambda/(2k)})^{N\theta},
\end{equation}
and applying Lemma \ref{Lemma4.2} in the case where $\lambda(s,k)=\Lambda$ gives
\begin{equation}\label{2ndpfbd}
\widetilde{K}_{a_N,b_N}^{r_N}(\bm{\mathfrak{a}})\ll q^{H\epsilon}.
\end{equation}
By our hierarchy (\ref{6.2}), in combination with (\ref{1stpfbd}) and (\ref{2ndpfbd}), we may assume that $H\epsilon\leq\theta$, so that
\begin{equation*}
q^{-2s\theta}\ll (q^{\rho-N\Lambda/(2k)})^{\theta}.
\end{equation*}
Using the fact that $\rho<1$, we rearrange this to obtain
\begin{equation}\label{this}
q^{4s\theta}\gg (q^{N\Lambda/(2k)})^{\theta}.
\end{equation}
We now observe that (\ref{nutheta}) implies that $q^{\theta}$
is sufficiently large with respect to $s$, $k$ and $\Lambda$, so (\ref{this}) can only hold if $4s \geq N\Lambda/(2k)$.
The definition of N leads ultimately to the relation
\begin{equation*}
\Lambda\leq 8sk/N\leq\Lambda/2,
\end{equation*}
a contradiction to the assumption that $\lambda(s,k)=\Lambda>0$, and so Theorem \ref{Thm3.1} is proved. \qed

\section{Proof of Theorem \ref{Thm1.1}}\label{PfMainThm}
We follow the approach given in \cite[§11-12]{NEC} with appropriate modifications for our ellipsephic situation. Let $\tau>0$ be sufficiently small in terms of $s$ and $k$, and let $X$ be sufficiently large in terms of $\tau,s,k$ and $\epsilon$. Set
\begin{align*}
B=\bigg\lceil\frac{k\log{X}}{\log p}\bigg\rceil,
\end{align*}
so that we have $X\leq p^{B/k}\leq pX$. Recall that the Wronskian of our system of polynomials is given by
\begin{align*}
W(z,\bm{\phi})&=\det\big(\phi_j^{(i)}(z)\big)_{1\leq i,j\leq k}
\end{align*}
where $\phi_j^{(i)}(z)$ is the $i$th derivative of $\phi_j$ with respect to $z$. We partition our variables according to whether or not the Wronskian at that point is zero or non-zero modulo $p$. Let
\begin{align*}
F(\bm{\alpha};X)&:=\rho_0^{-1}\sum_{x\in\mathcal{E}(X)}\mathfrak{a}_x e(\psi(x;\bm{\alpha}))\\
&=\rho_0^{-1}\sum_{\substack{x\in\mathcal{E}(X)\\ W(x,\bm{\phi})\not\equiv 0}}\mathfrak{a}_x e(\psi(x;\bm{\alpha})) +\rho_0^{-1}\sum_{\substack{x\in\mathcal{E}(X)\\ W(x,\bm{\phi})\equiv 0}}\mathfrak{a}_x e(\psi(x;\bm{\alpha}))\\
&=: F_p(\bm{\alpha};X) + F_0(\bm{\alpha};X).
\end{align*}
We therefore wish to bound the expression
\begin{align*}
\oint\abs{F(\bm{\alpha};X)}^{2s}\,d\bm{\alpha}
&=\oint\abs{F_p(\bm{\alpha};X) + F_0(\bm{\alpha};X)}^{2s}\,d\bm{\alpha}\\
&\ll \oint\abs{F_p(\bm{\alpha};X)}^{2s}\,d\bm{\alpha}+\oint\abs{F_0(\bm{\alpha};X)}^{2s}\,d\bm{\alpha}.
\end{align*}

Note that wherever necessary, we may assume without loss of generality that the above sums are over all $x\in\mathcal{E}$, with $\mathfrak{a}_x$ taken to be $0$ whenever $x>X$ and whenever $x$ fails the relevant Wronskian condition.
 
We first tackle the situation in which $W(x,\bm{\phi})\not\equiv 0\mmod{p}$ for each of our variables $x$, and observe that
\begin{align*}
\oint\abs{F_p(\bm{\alpha};X)}^{2s}\,d\bm{\alpha}\ll \oint_{p^B}\abs{F_p(\bm{\alpha};X)}^{2s}\,d\bm{\alpha},
\end{align*}
since any solution to our system of equations is certainly a solution to the corresponding system of congruences modulo $p^B$. We therefore focus on bounding the latter expression.

Taking $c=\lceil\tau B\rceil$, we partition our exponential sum into arithmetic progressions modulo $p^c$, applying Lemma \ref{Lemma6.2} to see that
\begin{align*}
\rho_{0}^2 \left|F_p(\bm{\alpha};X)\right|^{2s}\leq q^{sc}\sum_{\substack{\xi\in\mathcal{E}(p^{c})}}\rho_{c}(\xi)^2 \left|f_{c}(\bm{\alpha},\xi)\right|^{2s}.
\end{align*}
Consequently, we have
\begin{align}\label{eq11.5}
\oint_{p^B}\abs{F_p(\bm{\alpha};X)}^{2s}\,d\bm{\alpha}\leq q^{sc}\rho_0^{-2}\sum_{\xi\in\mathcal{E}(p^c)}\rho_c(\xi)^2 \oint_{p^B}\abs{f_c(\bm{\alpha},\xi)}^{2s}\,d\bm{\alpha}.
\end{align}
Observe that the integral on the right-hand side counts ellipsephic solutions to the system of congruences
\begin{align}\label{theseones}
\sum_{1\leq i\leq s} \phi_j(p^c u_i+\xi) \equiv \sum_{1\leq i\leq s} \phi_j(p^c v_i+\xi) \mmod{p^B}\quad (1\leq j\leq k),
\end{align}
with the appropriate weights. In particular, the classes $\xi$ for which the weights are non-zero necessarily have $W(\xi,\bm{\phi})\not\equiv 0\mmod{p}$.

We restrict our attention to some fixed such $\xi\in\mathcal{E}({p^c})$, and apply a Taylor expansion around $\xi$ to the functions $\phi_j$ to obtain
\begin{align*}
\phi_j(p^cu+\xi)-\phi_j(\xi)
&=\sum_{l\geq 1}\frac{\phi_j^{(l)}(\xi)}{l!}(p^cu)^{l}\\
&\equiv \sum_{1\leq l\leq k}\omega_{lj}(p^cu)^l +(p^cu)^{k+1}\Phi_j(p^cu)\mmod{p^B}
\end{align*}
for suitable polynomials $\Phi_j(z)\in\mathbb{Z}[z]$ and integer coefficients $\omega_{lj}={\phi_j^{(l)}(\xi)}/{l!}$. Let $\Omega=(\omega_{lj})_{1\leq l,j\leq k}$, and observe that
\begin{align*}
\det{\Omega}=W(\xi,\bm{\phi})\bigg(\prod_{1\leq l\leq k}l!\bigg)^{-1}.
\end{align*}
Since $p>k$ and $W(\xi,\bm{\phi})\not\equiv 0\mmod{p}$, we see that $\Omega$ has an inverse $\Omega^{-1}$ modulo $p^B$ with integer coefficients. We may therefore take linear combinations of the congruences (\ref{theseones}) in order to replace the polynomials $\bm{\phi}$ with $\Omega^{-1}\bm{\phi}$ and $\bm{\Phi}$ with $\Omega^{-1}\bm{\Phi}$, so without loss of generality we can assume that $\Omega$ is the $k\times k$ identity matrix. Consequently, we have
\begin{align*}
\phi_j(p^cu+\xi)-\phi_j(\xi)
&\equiv (p^cu)^j +(p^cu)^{k+1}\Phi_j(p^cu)\mmod{p^B}\\
&\equiv (p^c)^j\big(u^j+(p^c)^{k+1-j}u^{k+1}\Phi_j(p^cu)\big)\mmod{p^B}
\end{align*}
Let $\Psi_j(z)=z^j+(p^c)^{k+1-j}z^{k+1}\Phi_j(p^cz)$ for $1\leq j\leq k$. Then $\bm{\Psi}$ is a $p^c$-spaced system with the property that
\begin{align*}
\sum_{1\leq i\leq s} (p^c)^j(\Psi_j(u_i)-\Psi_j(v_i))\equiv 0\mmod{p^B},\quad (1\leq j\leq k).
\end{align*}
Our choice of parameters ensures that $ck<B$, so we have
\begin{align*}
\sum_{1\leq i\leq s} (\Psi_j(u_i)-\Psi_j(v_i))\equiv 0\mmod{p^{B-ck}},\quad (1\leq j\leq k).
\end{align*}

We now follow a similar process to that used in Lemma \ref{Lemma7.1}. Consider the weights $\mathfrak{c}_u(\bm{\alpha})=\mathfrak{a}_{p^cu+\xi}\,e\big(\psi(p^cu+\xi;\bm{\alpha})\big)$, and observe that $\rho_0 (\mathfrak{c})=\rho_c(\xi)$. 
 Let
\begin{align*}
g_{\bm{\mathfrak{c}}}(\bm{\alpha},\bm{\beta})= \rho_0(\bm{\mathfrak{c} })^{-1}\sum_{u\in\mathcal{E}}\mathfrak{c}_u(\bm{\alpha})e\big(\beta_1\Psi_1(u)+\dots+\beta_k\Psi_k(u)\big)
\end{align*}
so that by (\ref{Udefn}) we have
\begin{align*}
\oint_{p^{B}}\abs{g_{\bm{\mathfrak{c}}}(\bm{\alpha},\bm{0})}^{2s} \,d\bm{\alpha}
&=\oint_{p^{B}}\oint_{p^{B-ck}} \left|g_{\bm{\mathfrak{c}}}(\bm{\alpha},\bm{\beta})\right|^{2s}\,d\bm{\beta}\,d\bm{\alpha} \\
&=\oint_{p^{B}}U_{s,k}^{B-ck,\bm{\Psi}}(\bm{\mathfrak{c}}) \,d\bm{\alpha}.
\end{align*}
Since the system $\bm{\Psi}$ is $p^c$-spaced, we may apply Corollary \ref{Cor3.2} to obtain
\begin{align*}
U_{s,k}^{B-ck}(\bm{\mathfrak{c}})\ll q^{H\epsilon}U_{s,k}^{B-ck,H}(\bm{\mathfrak{c}})
\end{align*}
for $H=\lceil {(B-ck)}/{k}\rceil=\lceil {B}/{k}\rceil-c$. Consequently, we see that
\begin{align*}
\oint_{p^{B}} \left|f_{{{c}}}(\bm{\alpha},\xi)\right|^{2s}\,d\bm{\alpha}=\oint_{p^{B}} \left|g_{\bm{\mathfrak{c}}}(\bm{\alpha},\bm{0})\right|^{2s}\,d\bm{\alpha}\ll q^{H\epsilon}U_{s,k}^{B-ck,H}(\bm{\mathfrak{c}}),
\end{align*}
and substituting this into (\ref{eq11.5}) yields
\begin{align}\label{yields}
\oint_{p^B}\abs{F_p(\bm{\alpha};X)}^{2s}\,d\bm{\alpha}\leq q^{sc}\rho_0^{-2}\sum_{\xi\in\mathcal{E}(p^c)}\rho_c(\xi)^2 q^{H\epsilon}U_{s,k}^{B-ck,H}(\bm{\mathfrak{c}}).
\end{align}

Let
\begin{align*}
g_{\bm{\mathfrak{c}},H}(\bm{\alpha},\bm{\beta},\eta)= \rho_H(\eta;\bm{\mathfrak{c} })^{-1}\sum_{\substack{u\in\mathcal{E}\\ u\equiv\eta\mmod{p^H}}} \mathfrak{c}_u(\bm{\alpha}) e\big(\beta_1\Psi_1(u)+\dots+\beta_k\Psi_k(u)\big).
\end{align*}
Recall that, by assumption, the coefficients $\mathfrak{c}_u(\bm{\alpha})$ are zero whenever $p^cu+\xi>X$. By Cauchy's inequality, we have
\begin{align*}
\rho_H(\eta;\bm{\mathfrak{c} })^{2}\abs{g_{\bm{\mathfrak{c}},H}(\bm{\alpha},\bm{\beta},\eta)}^2&=
\bigg(\sum_{\substack{p^cu+\xi\in\mathcal{E}(X)\\ u\equiv\eta\mmod{p^H}}} 1\bigg)
\sum_{\substack{u\in\mathcal{E}\\ u\equiv\eta\mmod{p^H}}}  \abs{\mathfrak{c}_u(\bm{\alpha})}^2\\
&\ll \bigg(1+\frac{X}{q^{c+H}}\bigg)\rho_H(\eta;\bm{\mathfrak{c} })^{2},
\end{align*}
and so
\begin{align*}
\abs{g_{\bm{\mathfrak{c}},H}(\bm{\alpha},\bm{\beta},\eta)}^2
&\ll 1+\frac{X}{q^{c+H}}.
\end{align*}
Consequently, summing over the classes $\eta\in\mathcal{E}({p^H})$, we obtain
\begin{align*}
U_{s,k}^{B-ck,H}(\bm{\mathfrak{c}})&=\rho_0(\bm{\mathfrak{c}})^{-2}\sum_{\eta\in\mathcal{E}(p^H)}\rho_H(\eta;\bm{\mathfrak{c}})^2\oint_{p^{B-ck}}\abs{g_{\bm{\mathfrak{c}},H}(\bm{\alpha},\bm{\beta},\eta)}^{2s}\,d\bm{\beta}\\
&\ll \bigg(1+\frac{X}{q^{c+H}}\bigg)^s.
\end{align*}
We substitute this into (\ref{yields}) to see that
\begin{align*}
\oint_{p^B}\abs{F_p(\bm{\alpha};X)}^{2s}\,d\bm{\alpha}&\leq q^{sc+H\epsilon} (1+X/q^{c+H})^s\\
&\ll q^{(2s\tau+\epsilon)B} (1+X/q^{B/k})^s
\end{align*}
by our choices of $c$ and $B$. Since $\tau$ is assumed to be sufficiently small in terms of $s$ and $k$, we obtain
\begin{align*}
\oint_{p^B}\abs{F_p(\bm{\alpha};X)}^{2s}\,d\bm{\alpha}&\ll q^{B\epsilon} (1+X/q^{B/k})^s
\end{align*}
for any $\epsilon>0$, and therefore
\begin{align}\label{Fpbound}
\oint_{p^B}\abs{F_p(\bm{\alpha};X)}^{2s}\,d\bm{\alpha}
\ll q^{B\epsilon/(4k)}\ll X^{(k+1)\epsilon/4k}\ll X^{\epsilon},
\end{align}
using our choice of $X$.

We now consider the contribution from $F_0(\bm{\alpha};X)$; namely, the situation in which $W(x,\bm{\phi})\equiv 0\mmod{p}$ for all of our variables $x$. If in fact we have $W(x,\bm{\phi})= 0$, then the number of choices for $x$ is bounded by the degree of ${W}$ (since $W$ is not identically zero) and is therefore $O(1)$. Otherwise, we have $W(x,\bm{\phi})\neq 0$, but $W(x,\bm{\phi})\equiv 0\mmod{p}$. Our assumption that $p$ is sufficiently large in terms of $\bm{\phi}$ ensures that $W$ is not identically zero as a polynomial modulo $p$, and we consequently have $O(1)$ choices for $x$ modulo $p$.  (When $\bm{\phi}$ is the Vinogradov system, we can observe directly that the condition $p>k$ is sufficient to ensure that $W$ does not vanish identically modulo $p$, which confirms the corresponding claim in the introduction to this paper.)

For compactness of notation, all congruences involving the Wronskian in the following expressions are modulo $p$. We have
\begin{align*}
\rho_0 F_0(\bm{\alpha};X) &= \sum_{\substack{x\in\mathcal{E}(X)\\ W(x,\bm{\phi})\equiv 0}} \mathfrak{a}_xe(\psi(x;\bm{\alpha}))
=\sum_{\substack{\xi\mmod{p}\\W(\xi,\bm{\phi})\equiv 0}}\rho_1(\xi)f_1(\bm{\alpha},\xi)
\end{align*}
and may apply H\"older's inequality to see that
\begin{align*}
\abs{F_0(\bm{\alpha};X)}^{2s} 
&\leq \rho_0^{-2s} \bigg(\sum_{\substack{\xi\mmod{p}\\W(\xi,\bm{\phi})\equiv 0}} 1\bigg)^{2s-1} \sum_{\substack{\xi\mmod{p}\\W(\xi,\bm{\phi})\equiv 0}}\rho_1(\xi)^{2s}\abs{f_1(\bm{\alpha},\xi)}^{2s}\\
&\ll \max_{\substack{\xi\mmod{p}\\W(\xi,\bm{\phi})\equiv 0}}\abs{f_1(\bm{\alpha},\xi)}^{2s}.
\end{align*}
Fixing $\xi$ to maximise the expression on the right-hand side, we can write
\begin{align}\label{xisystem}
\oint\abs{F_0(\bm{\alpha};X)}^{2s}\,d\bm{\alpha}\ll \oint \abs{f_1(\bm{\alpha},\xi)}^{2s}\,d\bm{\alpha}.
\end{align}
In other words, we may now assume that we are in the case in which $\bm{x}\equiv\bm{y}\equiv\xi\mmod{p}$. Writing $x_i=pu_i+\xi$ and $y_i=pv_i+\xi$, we see that the integral on the right-hand side of (\ref{xisystem}) counts ellipsephic solutions, with appropriate weights, to the system
\begin{align*}
\sum_{i=1}^s \phi_j(pu_i+\xi)=\sum_{i=1}^s \phi_j(pv_i+\xi), \quad (1\leq j\leq k),
\end{align*}
and by expanding the polynomials $\phi_j$, we may cancel one or more factors of $p$ through the system. We take linear combinations of the resulting equations to obtain a system $\Phi_j\mmod{p}$ of lower degree. If this system has vanishing Wronskian as a polynomial in $\bm{u}$ and $\bm{v}$, we may eliminate further terms until we are no longer in that situation. Consequently, we assume that our new system has non-vanishing Wronskian. We then repeat the procedure of this section by dividing solutions into those for which the Wronskian vanishes modulo $p$, and those for which it does not, and handling the two cases separately.

This process terminates after a number of steps which is $O(1)$, confirming that $\oint \abs{F_0(\bm{\alpha};X)}^{2s}\,d\bm{\alpha}\ll X^{\epsilon}$. Combining this with (\ref{Fpbound}), we obtain 
\begin{align*}
\oint \abs{F(\bm{\alpha};X)}^{2s}\,d\bm{\alpha}\ll X^{\epsilon}
\end{align*}
which yields Theorem \ref{Thm1.1}.

\newcommand{\noop}[1]{}

\end{document}